\documentclass[12pt]{article}
\usepackage[utf8]{inputenc}
\usepackage{geometry}
\usepackage{amsthm}
\usepackage{amsmath}
\usepackage{amssymb}
\usepackage{xfrac}
\usepackage{graphicx}
\usepackage{enumerate}
\usepackage{hyperref}
\hypersetup{
	colorlinks   = true,    % Colours links instead of ugly boxes
	urlcolor     = blue,    % Colour for external hyperlinks
	linkcolor    = blue,    % Colour of internal links
	citecolor    = red      % Colour of citations
}

\DeclareMathOperator{\arccot}{arccot}
\makeatletter
\DeclareMathOperator{\Arccos}{Arccos}
\DeclareMathOperator{\Log}{Log}

\AtBeginDocument{\DeclareFontEncoding{T2A}{}{}}

\numberwithin{figure}{section}
\theoremstyle{plain}

\theoremstyle{plain}

\@ifundefined{date}{}{\date{}}

\usepackage[english]{babel}
\newtheorem{theorem}{Theorem}[section]

\newtheorem{lm}{Lemma}[section]

\numberwithin{equation}{section}
\usepackage{authblk}
\usepackage{babel}
\providecommand{\propositionname}{Proposition}
\providecommand{\theoremname}{Theorem}
\title{Asymptotics of the eigenvalues of seven-diagonal Toeplitz matrices of a special form}
\author[1,2,3]{V. Stukopin}
\author[4,5]{S. Grudsky}
\author[1,2]{I. Voronin}
\author[4]{M. Barrera}
\affil[1]{MCCME (Moscow Center for Continuous Mathematical Education)}
\affil[2]{MIPT (Moscow Institute of Physics and Technology)}
\affil[3]{SMI of VSC RAS (South Mathematical Institute of Vladikavkaz Scientific Center of Russian Academy of Sciences)}
\affil[4]{Departamento de Matemáticas, CINVESTAV del IPN, Ciudad de México, México}
\affil[5]{Regional Mathematical Center of the Southern Federal University, Rostov-on-Don, Russia}
\makeatother

\begin{document}
\maketitle %\maketitle
\begin{abstract}

We find uniform asymptotic formulas for all the eigenvalues of certain 7-diagonal symmetric Toeplitz matrices of large dimension.
The entries of the matrices are real and we consider the case where the real-valued generating function such that its first
five derivatives at the one endpoint of interval are equal zero. This is not the simple-loop case considered earlier. We obtain nonlinear equations for the eigenvalues. It should be noted that our equations have a more complicated structure than the equations for the simple loop case.

{\bf Keywords} Toeplitz matrices, eigenvectors, asymptotic expansions.

\end{abstract}

%\newpage
\section{Introduction}
Let $a(t)$ be a Lebesgue integrable function defined on the unit circle $\mathbb{T} = \{t \in \mathbb{C}: |t|=1\}$. We denote by $T_n(a)$ the Toeplitz matrix $T_n(a) := (a_{j-k})_{j,k=1}^{n-1}$, where $n \in \mathbb{N}$ is a natural number, and $a_l$ denotes the $l$-th coefficient of the Fourier series of the function $a$. Note that the Toeplitz matrix can be viewed as an operator from a finite dimensional vector space.  The function $a(t)$ is called the symbol of the Toeplitz matrix (Toeplitz operator) $T_n(a)$.
This paper is devoted to finding asymptotic formulas for the eigenvalues of the Toeplitz matrix with the symbol $a(t)=(t-2+\frac{1}{t})^3$.

Toeplitz matrices, as well as closely related Toeplitz operators, have been intensively studied for various classes of symbols over the past, about a hundred years (\cite{GS}, \cite{SS}, \cite{WiOT}, \cite{BG}, \cite{Uni}). The importance of this subject is largely due to the numerous applications of Toeplitz matrices in numerical methods of differential and integral equations, probability theory, statistical physics (see, for example, \cite{DIKIsing}, \cite{DIK}, \cite{Kad}, \cite{MW}). As mentioned above, this work is devoted to finding asymptotic formulas for the eigenvalues of the Toeplitz matrix with the symbol $a(t)=(t-2+\frac{1}{t})^3$. Toeplitz matrices with this symbol are self-adjoint matrices. However, the study of non-self-adjoint Toeplitz matrices can also be reduced to this case, the symbol of which is the cube of the linear Laurent polynomial and has a third-degree derivative at the end of the interval equal to zero. We note that all the asymptotic formulas for the eigenvalues obtained in this paper, in essence, admit an eigenvalue uniform with respect to the number, an estimate for the remainder term. It should be said that the symbol under consideration has specific properties: it is a real, symmetric function, and the first and second derivatives of the symbol vanish at the points $t= \pm 1$. The last condition, namely the vanishing of the second derivative, significantly complicates the problem of finding an asymptotic formula for the eigenvalues, since in this case the general research methods developed in the work \cite{BGS}   are inapplicable (see also works \cite{BBG}, \cite{BBGM1}, \cite{Grudsky2011}, \cite{BGM1}, \cite{Bat2019}, which present general approaches to finding the asymptotics of the eigenvalues for various classes of Toeplitz matrices). In addition, the case we are considering is more complicated than that considered in the work \cite{BarreraG}. The main idea of the study is that due to the results of Spitzer and Schmidt  (\cite{SS}) the eigenvalues should be sought on the limiting spectrum, which is obtained from the condition of the coincidence of the moduli of the mean roots of the function $b(t, \lambda)= a(t)-\lambda$. In our case, there are six such roots: $\xi_i(\lambda), i=1,...,6$ and the limiting spectrum is determined by the condition  $|\xi_3(\lambda)| = |\xi_4(\lambda)|$. In the self-adjoint case, the limiting spectrum is a interval and the eigenvalues lie on this interval, which somewhat simplifies the problem. But in the case of the aforementioned symbol at the endpoints, the moduli of all the roots coincide, and this fact significantly complicates the problem of finding the asymptotic formula. Moreover, if you understand what the asymptotic formula looks like in the case of multiple roots of the function mentioned above $b(t, \lambda)$, this will be an important step in finding such a formula for the eigenvalues of an arbitrary band Toeplitz matrix. Thus, the asymptotics of the eigenvalues of Toeplitz matrices with the symbol  $a(t)=(t-2+\frac{1}{t})^3$ cannot be derived from the results known to us in this area, and the solution of such a problem is of interest in view of the emerging fundamental difficulties, the resolution of which will be an important step in the study of the general problem of finding the asymptotics of the spectrum of an arbitrary banded Toeplitz matrix.
\vspace{1cm}
\section{Main results}
In this section, we will present the main results of the article. We formulate a theorem describing an asymptotic formula for the eigenvalues of a Toeplitz matrix with the symbol mentioned above. The eigenvalues are calculated as the values of the function $g(\varphi) = a(e^{\mathrm{i}\varphi})$, where $a(t) = (t-2+\frac{1}{t})^3$, for fixed values of the argument $\varphi$. Furthermore, the function $g(\varphi)=a(e^{i\varphi})=-\left(2\sin\frac{\varphi}{2}\right)^6$ defined on $[0,2\pi]$ has the following properties:
\begin{enumerate}[$(i)$]
	\item The function $g:[0,2\pi]\to\mathbb{R}$, has range $[m,0]$ with $m<0$.
	\item $g(\pi):=m$, $g^{(1)}(\pi)=0$, and $g^{(2)}(\pi)>0$.
	\item $g(0)=g(2\pi)=0$, $g^{(k)}(0)=g^{(k)}(2\pi)=0$ $(k=1,\dots,5)$, and $g^{(6)}(0)=g^{(6)}(2\pi)<0$.
\end{enumerate}
Thus, the structure of the asymptotic formula for the eigenvalues is such that this formula is a refinement, on the one hand, of Szego's limit theorem, which describes the limit spectrum of Toeplitz matrices as the image of the unit circle $\mathbb{T}$ under the action of the symbol, and on the other, as mentioned in the introduction, is a refinement of the results of Spitzer and Schmidt (\cite{SS}), which give the same answer in self-adjoint as Szego's limit theorem (\cite{GS}, \cite{Uni}).

Note that the problem is solved with respect to the variable $\varphi$, from which the eigenvalues $\lambda$ are expressed by a simple substitution $\lambda=g (\varphi)$. Let's introduce some functions. All functions will be defined on the interval $ \varphi \in (0,\pi)$.
\begin{equation}
\beta(\varphi):=\arccos{(1-(1-\cos{\varphi})e^{\frac{2 \pi i}{3}})}
\label{opred beta}
\end{equation}
$\Arccos$ is multivalued function, $\beta(\varphi)$ is one of its regular branches. The existence of this branch when $ \varphi \in (0,\pi)$ will be shown in the section \ref{Razdel pol Cheb}
\begin{equation}
\begin{aligned}
c(\varphi):=\Re{(\beta(\varphi))}, & & & b(\varphi):=\Im{(\beta(\varphi))}\\
B(\varphi):=\Re{(\sin{(\beta)}e^{\frac{-\pi i}{3}})}, & & & C(\varphi):=-\Im{(\sin{(\beta)}e^{\frac{-\pi i}{3}})}
\end{aligned}
\end{equation}
\begin{theorem}
	\label{Gl_theor}
	Let $\lambda=g(\varphi)$. Then the equation $\det{T_n(a-g (\varphi))}=0$ is equivalent to the following equations:
	\begin{equation}
	\tan{\left(\frac{n+3}{2}\varphi\right)}=f(\varphi),
	\label{Glavnoe ur 1}
	\end{equation}
	and
	\begin{equation}
	\tan{\left(\frac{n+3}{2}\varphi\right)}=\frac{1}{h(\varphi)},
	\label{Glavnoe ur 2}
	\end{equation}
	where  $$f(\varphi)=2\frac{B(\varphi)\sin{((n+3)c(\varphi))}+C(\varphi)\sinh{((n+3)b(\varphi))}}{\sin{(\varphi)}(\cos{((n+3)c(\varphi))}+\cosh{((n+3)b(\varphi))})},$$
	$$h(\varphi)=2\frac{B(\varphi)\sin{((n+3)c(\varphi))}-C(\varphi)\sinh{((n+3)b(\varphi))}}{\sin{(\varphi)}(-\cos{((n+3)c(\varphi))}+\cosh{((n+3)b(\varphi))})}, $$
	$\varphi \in (0, \pi)$.
\end{theorem}
It is ease to see that equations \ref{Glavnoe ur 1} and \ref{Glavnoe ur 2} are equivalent the following set of equations correspondingly 
\begin{equation}
	\label{sl ur 1}
	\varphi = \frac{2}{n+3}\left[\pi j + \arctan f(\varphi) \right],
\end{equation}
$$
j \in \left\{1,2,\dots ,\left[\frac{n+1}{2}\right]\right\}
$$
and
\begin{equation}
	\label{sl ur 2}
	\varphi = \frac{2}{n+3}\left[\pi j + \frac{\pi}{2}-\arctan h(\varphi) \right],
\end{equation}
$$
j \in \left\{1,2,\dots ,\left[\frac{n}{2}\right]\right\}
$$
Apply to a solution of this equations Fix Point Method. Put
$$
\varphi_{2 j-1}^{(0)}=d_{2 j-1}, \;\;\;	\varphi_{2j-1}^{(k+1)}=\frac{2}{n+3}\left[\pi j +\arctan{(f(\varphi_{2j-1}^{(k)}))}\right]
$$ 
and 
$$
\varphi_{2 j}^{(0)}=d_{2 j}, \;\;\;	\varphi_{2j}^{(k+1)}=\frac{2}{n+3}\left[\pi j+\frac{\pi}{2} -\arctan{(h(\varphi_{2j}^{(k)}))}\right]
$$
where
$$
d_m=\frac{\pi (m+1)}{n+3}, \;\;\; m=1,2,\dots ,n
$$
\begin{theorem}
	\label{theorem 2}
	If $n$ is sufficiently large then \\
	1)The equation (\ref{Glavnoe ur 1}) has exactly one root $ \varphi_{2j-1}$ on each of the intervals $(\frac{\pi(2j-1)}{n+3},\frac{\pi(2j+1)}{n+3})$, where $j \in \{1,\dots ,[\frac{n+1}{2}]\}$. Moreover, we can write the following estimate:
	\begin{equation}
	\left| \varphi_{2j-1}^{(k+1)} - \varphi_{2j-1}^{(k+1)} \right| \le \frac{L}{(n+3)} (0.8)^k,
	\label{Rekurent formula 1}
	\end{equation}
	where $k$ - iteration number, $L$ does not depend on $j$ and $n$.\\
	2) The equation (\ref{Glavnoe ur 2}) has exactly one root $\varphi_{2j}$ on each of the intervals $(\frac{2\pi j}{n+3},\frac{2\pi(j+1)}{n+3})$, where $j \in \{1,\dots ,[\frac{n}{2}]\}$. Moreover,we can write the following estimate:
	\begin{equation}
	\left| \varphi_{2j}^{(k+1)} - \varphi_{2j}^{(k+1)} \right| \le \frac{L}{(n+3)} (0.8)^k,
	\label{Rekurent formula 2}
	\end{equation}
	where $k$ - iteration number, does not depend on $j$ and $n$.\\
\end{theorem}
\begin{lm}
	For a sufficiently large $n$, the roots of the equations (\ref{Glavnoe ur 1}) and (\ref{Glavnoe ur 2}) are pairwise distinct.
	\label{lemma2}
\end{lm}
This theorem shows that the equations (\ref{Glavnoe ur 1}) and (\ref{Glavnoe ur 2}) have at least $\left[\frac{n+1}{2}\right]+\left[\frac{n}{2}\right]=n $ roots, however, since the set of these roots coincides with the set of eigenvalues of the $n\times n$ matrix, so there cannot be more roots. The problem was reduced to solving the equations
\begin{equation}
\varphi=\frac{2}{n+3}\left[\pi j +\arctan{f(\varphi)}\right]
\label{novoe glavnoe ur 1}
\end{equation}
on intervals $(\frac{\pi(2j-1)}{n+3},\frac{\pi(2j+1)}{n+3})$, where $j \in \{1,\dots ,[\frac{n+1}{2}]\}$, and
\begin{equation}
\varphi=\frac{2}{n+3}\left[\pi j+\frac{\pi}{2}-\arctan{h(\varphi)}\right]
\label{novoe glavnoe ur 2}
\end{equation}
on intervals  $(\frac{2\pi j}{n+3},\frac{2\pi(j+1)}{n+3})$, where $j \in \{1,\dots ,[\frac{n}{2}]\}$\\
Let $q:=\frac{n+3}{2}$. To solve the equation (\ref{novoe glavnoe ur 1}) we introduce the parameter $d_{1,j}:=d_{2j-1}=\frac{2 \pi j}{n+3}$. Then $\varphi$ can be represented as $\varphi=d_{1,j}+\frac{u}{q}$, and equation \eqref{novoe glavnoe ur 1} can be rewritten as:
\begin{equation}
u=\arctan{f\left(d_{1,j}+\frac{u}{q}\right)}
\label{Glavnoe ur U}
\end{equation}
where $u \in (-\frac{\pi}{2}, \frac{\pi}{2})$.
Similarly, to solve the equation \ref{novoe glavnoe ur 2}, we introduce the parameter $d_{2,j}:=d_{2j}=\frac{2 \pi j+1}{n+3}$, so if $\varphi=d_{2,j}+\frac{w}{q}$, then equation \eqref{novoe glavnoe ur 2} can be rewritten as:
\begin{equation}
w=-\arctan{h\left(d_{2,j}+\frac{w}{q}\right)}
\label{Glavnoe ur W}
\end{equation}
where  $w \in (-\frac{\pi}{2}, \frac{\pi}{2})$
\begin{theorem}
	\label{Teorema sluch 1}
	Let $a(t)=(t-2+\frac{1}{t})^3$. Then, as $n \rightarrow \infty$
	\begin{enumerate}
		\item If  $d_{1,j}: \:\frac{1}{2}e^{\pi(j-1)}>q^2$ then:
		\begin{equation}
		\varphi_{2j-1}=d_{1,j}+\frac{2u_1^{\star}}{n+3}+\frac{4u_2^{\star}}{(n+3)^2}+O\left(\frac{1}{n^3}\right),
		\label{koren phi1}
		\end{equation}
		where $u_1^{\star}=\arctan{\left(2\frac{C(d_{1,j})}{\sin{(d_{1,j})}}\right)}$  and $u_2^{\star}=2\frac{C'(d_{1,j})\sin{(d_{1,j})}-C(d_{1,j})\cos{(d_{1,j})}}{\sin^2{(d_{1,j})}+4C^2(d_{1,j})}\arctan{\left(2\frac{C(d_{1,j})}{\sin{(d_{1,j})}}\right)}$
		\item If  $d_{2,j}: \:\frac{1}{2}e^{\pi(j-1)}>q^2$ then:
		\begin{equation}
		\varphi_{2j}=d_{2,j}+\frac{2w_1}{n+3}+\frac{4w_2^{\star}}{(n+3)^2}+O\left(\frac{1}{n^3}\right),
		\label{koren phi2}
		\end{equation}
		where $w_1^{\star}=\arctan{\left(2\frac{C(d_{2,j})}{\sin{(d_{2,j})}}\right)}$  and $w_2^{\star}=2\frac{C'(d_{2,j})\sin{(d_{2,j})}-C(d_{2,j})\cos{(d_{2,j})}}{\sin^2{(d_{2,j})}+4C^2(d_{2,j})}\arctan{\left(2\frac{C(d_{2,j})}{\sin{(d_{2,j})}}\right)}$
	\end{enumerate}
\end{theorem}
For brevity, we define:
\begin{equation}
Z_1^{(1)}=2\frac{B_1^{(1)}\sin{(P_a^{(1)})}+C_1^{(1)}\sinh{(P_b^{(1)})}}{(\cos{(P_a^{(1)})}+\cosh{(P_b^{(1)})})}
\label{uravn dlya u1}
\end{equation}
and
\begin{equation}
Z_1^{(2)}=2\frac{B_1^{(2)}\sin{(P_a^{(2)})}-C_1^{(2)}\sinh{(P_b^{(2)})}}{(-\cos{(P_a^{(2)})}+\cosh{(P_b^{(2)})})},
\label{uravn dlya w1}
\end{equation}
where $P_a^{(1)}=qd_{1,j}+u_1 $, $P_b^{(1)}=\sqrt{3}(qd_{1,j}+u_1) $, $B_1^{(1)}=1+\frac{3}{16}d_{1,j}^2$, $C_1^{(1)}=\frac{\sqrt{3}}{16}d_{1,j}^2$
\\$P_a^{(2)}=qd_{2,j}+w_1 $, $P_b^{(2)}=\sqrt{3}(qd_{2,j}+w_1)$, $B_1^{(2)}=1+\frac{3}{16}d_{2,j}^2$, $C_1^{(2)}=\frac{\sqrt{3}}{16}d_{2,j}^2$
\begin{theorem}
	\label{theorem sluch 2}
	Let $a(t)=(t-2+\frac{1}{t})^3$. Then, as $n \rightarrow \infty$
	\begin{enumerate}
		\item If  $d_{1,j}: \:\frac{1}{2}e^{\pi(j-1)}\leq q^2$ then:
		\begin{equation}
		\varphi_{2j-1}=d_{1,j}+\frac{2u_1^{\star}}{n+3}+\frac{4u_2^{\star}}{(n+3)^2}+O\left(\frac{1}{n^3}\right),
		\end{equation}
		where $u_1^{\star}$ is the solution of equation  $u_1=\arctan{(Z^{(1)})}$  and $u_2^{\star}=R^{(1)}(u_1^{\star})$ (see proof of the theorem)
		\item If  $d_{2,j}: \:\frac{1}{2}e^{\pi(j-1)} \leq q^2$ then:
		\begin{equation}
		\varphi_{2j}=d_{2,j}+\frac{2w_1^{\star}}{n+3}+\frac{4w_2^{\star}}{(n+3)^2}+O\left(\frac{1}{n^3}\right),
		\end{equation}
		where $w_1^{\star}$ is the solution of equation $w_1=-\arctan{(Z^{(2)})}$  and $w_2^{\star}=R^{(2)}(w_1^{\star})$ (see proof of the theorem)
	\end{enumerate}
\end{theorem}
The following result gives us the asymptotic formulas for eigenvalues $\lambda$.
\begin{theorem}
	\label{posled teorema} Let $a(t)=(t-2+\frac{1}{t})^3$. Then, as $n \rightarrow \infty$
	\begin{enumerate}
		\item
		$$\lambda_{2j-1}^{(n)}=g(d_{1,j})+g'(d_{1,j})\frac{2u_1^{\star}}{n+3}+\frac{4u_2^{\star}g'(d_{1,j})+  2(u_1^{\star})^2g''(d_{1,j})}{(n+3)^2}+o\left(\frac{1}{n^2}\right),$$
		where $u_1^{\star}$ and $u_2^{\star}$ is defined in the same way as in the theorem \ref{Teorema sluch 1} if $d_{1,j}: \:\frac{1}{2}e^{\pi(j-1)}>q^2$, and in the same way as in the theorem \ref{theorem sluch 2} if $d_{1,j}: \:\frac{1}{2}e^{\pi(j-1)}\leq q^2$.
		\item
		$$\lambda_{2j}^{(n)}=g(d_{2,j})+g'(d_{2,j})\frac{2w_1^{\star}}{n+3}+\frac{4w_2^{\star}g'(d_{2,j})+2(w_2^{\star})^2g''(d_{2,j})}{(n+3)^2}+o\left(\frac{1}{n^2}\right),$$
		where $w_1^{\star}$ and $w_2^{\star}$ is defined in the same way as in the theorem \ref{Teorema sluch 1} if $d_{2,j}: \:\frac{1}{2}e^{\pi(j-1)}>q^2$, and in the same way as in the theorem \ref{theorem sluch 2} if $d_{2,j}: \:\frac{1}{2}e^{\pi(j-1)}\leq q^2$.
	\end{enumerate}
\end{theorem}
The following result gives us the asymptotic formulas for the extreme eigenvalues near zero.
\begin{theorem}\label{extreme eigenvalues}
	Let $g(\varphi) = a(e^{i\varphi})=-\left(2\sin\frac{\varphi}{2}\right)^6$.
	\begin{enumerate}[i)]
		\item If $d_{1,j}\to 0$ as $n\to\infty$, then
		\[\lambda_{2j-1}^{(n)}=-\frac{(2\pi j+2u_1^{*})^6}{(n+3)^6}-\frac{24u_2^*(2\pi j+2u_1^*)^5}{(n+3)^7}+\Delta_1(n,j),\]
		where $\vert\Delta_1(n,j)\vert\le M_1\left(\frac{d_{1,j}^5}{n^3}+d_{1,j}^{10}\right)$ and the constant $M_1$ does not depend in $j$ and $n$.
		\item If $d_{2,j}\to 0$ as $n\to\infty$, then
		\[\lambda_{2j}^{(n)}=-\frac{((2j+1)\pi+2w_1^{*})^6}{(n+3)^6}-\frac{24w_2^*((2j+1)\pi+2w_1^*)^5}{(n+3)^7}+\Delta_2(n,j),\]
		where $\vert\Delta_2(n,j)\vert\le M_2\left(\frac{d_{2,j}^5}{n^3}+d_{2,j}^{10}\right)$ and the constant $M_2$ does not depend in $j$ and $n$.
	\end{enumerate}
\end{theorem}
\section{Auxiliary results}
\label{Vspomogatelnie rezultati}	
In this section, some auxiliary statements will be proved. Now we will introduce auxiliary functions.
$$\alpha_2=\alpha_2(\varphi)=1+(\cos{\varphi}-1)e^{\frac{2\pi i}{3}}$$
\begin{equation}
B_c=B_c(\varphi)=|\alpha_2|
\label{opred B_c}
\end{equation}
\begin{equation}
\psi_c=\psi_c(\varphi)=\arg(\alpha_2)
\label{opred psi_c}
\end{equation}
  \begin{equation}
B_s=B_s(\varphi)=|(1-\alpha_2^2)^\frac{1}{2}|=\sqrt[4]{(1-\cos{\varphi})^2(7-4\cos{\varphi}+\cos^2{\varphi})}
\label{opred B_s}
\end{equation}
\begin{equation}
\psi_s=\psi_s(\varphi)=\arg{((1-\alpha_2^2)^\frac{1}{2})}=\frac{\pi}{2}+\frac{\arctan{\frac{\sqrt{3}(3-\cos{\varphi})}{(-1-\cos{\varphi})}}}{2}
\label{opred psi_s}
\end{equation}
Obviously, $\arg {((1- \alpha_2^2)^\frac{1}{2})}$ has two regular branches,we select one of them.\\
All subsequent statements will be proved when $\varphi \in (0, \pi)$.
\begin{enumerate}
	\item
	\begin{enumerate}
	\item \label{B_s} $B_s(\varphi)$ is increasing function, and $B_s(\varphi) \in (0, 2\sqrt[4]{3})$
	\item \label{psi_s} $\psi_s(\varphi)$ is decreasing function, and $\psi_s(\varphi) \in (\frac{\pi}{4},\frac{\pi}{3})$
	\item \label{B_c} $B_c(\varphi)$ is increasing function, and $B_c(\varphi) \in (1,\sqrt{7})$
	\item \label{psi_c} $\psi_c(\varphi)$ is decreasing function, and $\psi_c(\varphi) \in (-\arctan{\frac{\sqrt{3}}{2}},0)$
	\end{enumerate}
	\begin{proof}We differentiate the corresponding functions, and decompose them into multipliers. The values at the edges of the interval are found by simple substitution:
		\begin{equation}
		B_s'=\dfrac{\sin{(\varphi)}(9-7\cos{(\varphi)}+2\cos^2{(\varphi)})}{2(1-\cos{(\varphi)})(7-4\cos{(\varphi)}+\cos^2{(\varphi)})^{\frac{3}{4}}} > 0.
		\label{B_s'}
		\end{equation}
		So, the function $B_s (\varphi)$ is increasing.
			\begin{equation}
		\psi_s'=-\dfrac{\sqrt{3}\sin{(\varphi)}}{2(7-4\cos{(\varphi)}+\cos^2{(\varphi)})} < 0.
		\label{psi_s'}
		\end{equation}
		So, the function $\psi_s(\varphi)$ is decreasing.
			\begin{equation}
		B_c'=\dfrac{\sin{(\varphi)}(3-2\cos{(\varphi)})}{2\sqrt{(3-3\cos{(\varphi)}+3\cos^2{(\varphi)})}} > 0.
		\label{B_c'}
		\end{equation}
		So, the function $B_c(\varphi)$ is increasing.
			\begin{equation}
		\psi_c'=-\dfrac{\sqrt{3}\sin{(\varphi)}}{2(3-3\cos{(\varphi)}+\cos^2{(\varphi)})} < 0.\label{psi_c'}
		\end{equation}
		So, the function $\psi_c(\varphi)$ is decreasing.
	\end{proof}
	\item \begin{enumerate}		
	\item \label{Bc_cos-Bs_sin} $B_c\cos{(\psi_c)}- B_s\sin{(\psi_s)}>0$
	\item \label{Bc_cos+Bs_sin} $B_c\cos{(\psi_c)}+B_s\sin{(\psi_s)}>0$
	\end{enumerate}
	\begin{proof}
	 We show that $B_c>B_s$ and $ \cos{(\psi_c)}>\sin{(\psi_s)}$, from which the statement of this item will follow. Since $B_c>0$ and $B_s>0$, therefore $B_c>B_s$ is equivalent to $B_c^4-B_s^4>0$.
	 $$
	 B_c^4-B_s^4=2-\cos^2{\varphi}>0
	 $$
	Now we show that  $ \cos{(\psi_c)}>\sin{(\psi_s)}$. $\cos{(\psi_c)}=\sin{(\frac{\pi}{2}+\psi_c)}$, at the same time, from the points  (\ref{psi_s} and \ref{psi_c}), it follows that $(\frac{\pi}{2}+\psi_c)$ and $\psi_s$, are in the first quadrant, so $\sin{(\frac{\pi}{2}+\psi_c)}>\sin{(\psi_s)}$ is equivalent to $\frac{\pi}{2}+\psi_c-\psi_s>0$. It's not hard to get that
	$$\psi_c'-\psi_s'=-\dfrac{\sqrt{3}\sin{(\varphi)}(4-\cos{(\varphi)})}{2(7-4\cos{(\varphi)}+\cos^2{(\varphi)})(3-3\cos{(\varphi)}+\cos^2{(\varphi)})} < 0.$$
	Which means $\frac{\pi}{2}+\psi_c-\psi_s>\frac{\pi}{2}+\psi_c(\pi)-\psi_s(\pi)=\frac{\pi}{2}-\arctan{\frac{\sqrt{3}}{2}}-\frac{\pi}{4}>0$, so  $ \cos{(\psi_c)}>\sin{(\psi_s)}$. This means that the statement (\ref{Bc_cos-Bs_sin}) is true.
	From the statements (\ref{B_s} - \ref{psi_c}), it follows that both terms in the expression $B_c\cos{(\psi_c)}+B_s\sin{(\psi_s)}$ is positive, which means that the statement(\ref{Bc_cos+Bs_sin}) is also true.	
	\end{proof}
	\item \begin{enumerate}
	\item \label{a} $c(\varphi)$ is increasing function.
	\item \label{b} $b(\varphi)$ is  increasing function.
	\item \label{a'} $c'(\varphi)$ is decreasing function, and $c'(0)=\frac{1}{2}$, $c'(\pi)=0$.
	\item \label{b'} $b'(\varphi)$ is decreasing function, and $b'(0)=\frac{\sqrt{3}}{2}$, $b'(\pi)=0$.
\end{enumerate}
	\begin{proof}
	$$(\cos{\beta})'= -\beta '\sin{\beta},$$ so	$$\beta '=-\dfrac{(\cos{\beta})'}{\sin{\beta}},$$ from where it is not difficult to get that
	\begin{equation}
	\beta '=\dfrac{\sin{\varphi}}{\sin{\beta}}e^{\frac{2 \pi i}{3}}
	\label{beta'}
	\end{equation}
	Then $$c'(\varphi)=Re(\beta ')=\dfrac{\sin{\varphi}}{B_s}\cos{\left(\frac{2 \pi}{3}-\psi_s\right)},$$
	$$b'(\varphi)=Re(\beta ')=\dfrac{\sin{\varphi}}{B_s}\sin{\left(\frac{2 \pi}{3}-\psi_s\right)}.$$
	Since $\psi_s(\varphi) \in (\frac{\pi}{4},\frac{\pi}{3})$, then $c'(\varphi)>0$, $b'(\varphi)>0$ so the functions $a(\varphi)$ and $b(\varphi)$ are increasing.
	Let's find $ \beta''$.
	$$\beta''=\left(\dfrac{\sin{(\varphi)}}{\sin{(\beta)}}e^{\frac{2 \pi i}{3}} \right)'=\dfrac{\cos{(\varphi)}\sin{(\beta)}-\sin{(\varphi)}\cos{(\beta)}\beta'}{\sin^2{(\beta)}}e^{\frac{2 \pi i}{3}} $$
	Substituting the equality (\ref{beta'}) and reducing to a common denominator, we get the following equality:
	$$\beta''=\dfrac{\cos{(\varphi)}\sin^2{(\beta)}e^{\frac{2 \pi i}{3}}-\sin^2{(\varphi)}\cos{(\beta)}e^{-\frac{2 \pi i}{3}}}{\sin^3{(\beta)}}.$$
	Given that $\sin^2{t}=1-\cos^2{t}$ we get the equality:
	$$\beta''=\dfrac{\cos{(\varphi)}(1-\cos^2{(\beta)})e^{\frac{2 \pi i}{3}}-(1-\cos^2{(\varphi)})\cos{(\beta)}e^{-\frac{2 \pi i}{3}}}{\sin^3{(\beta)}}$$
	Substituting the equality (\ref{ur opred beta i gamma}) we get
	$$\beta''=\dfrac{\cos{(\varphi)}(1-(1-(1-\cos{(\varphi)})e^{\frac{2 \pi i}{3}})^2)e^{\frac{2 \pi i}{3}}-(1-\cos^2{(\varphi)})(1-(1-\cos{(\varphi)})e^{\frac{2 \pi i}{3}})e^{-\frac{2 \pi i}{3}}}{\sin^3{(\beta)}}$$
	Then expanding brackets and combining like terms we get that:
	\begin{equation}
	\beta''=\dfrac{\sqrt{3}(1-\cos{(\varphi)})^2}{\sin^3{(\beta)}}e^{\frac{\pi i}{6}}
	\label{beta''}
	\end{equation}
		Since $c''=\Re(\beta'')$, $b''=\Im(\beta'')$ and also that $\psi_s \in (\frac{\pi}{4},\frac{\pi}{3})$ we get
	$$c''=\dfrac{\sqrt{3}(1-\cos{(\varphi)})^2}{B_s^3}\cos{\left(\frac{\pi}{6}-3\psi_s\right)}<0,$$
	$$b''=\dfrac{\sqrt{3}(1-\cos{(\varphi)})^2}{B_s^3}\sin{\left(\frac{\pi}{6}-3\psi_s\right)}<0,$$
	so the functions $c'(\varphi)$ and $b'(\varphi)$ are decreasing.
		Let's find $c'(0)=\Re{(\beta'(0))}$.
	$$
	\beta'(0)=\lim_{\varphi \to 0}\dfrac{\sin{(\varphi)}}{B_s}e^{\left(\frac{2 \pi i}{3}-\psi_s\right)}
	$$
	Expand $\sin{(\varphi)}$ and $B_s$ in a Taylor series up to the first term is not difficult to get that
	$$\beta'(0)=e^{\frac{\pi i}{3}}.$$
	It follows that $c'(0)=\frac{1}{2}$ and $b'(0)=\frac{\sqrt{3}}{2}$. Well $c'(\pi)$ and $b'(\pi)$ can be found by a simple substitution.	
	\end{proof}
	\item \begin{enumerate}
	\item \label{a/phi} $\frac{c(\varphi)}{\varphi}$ is decreasing function.
	\item \label{b/phi} $\frac{b(\varphi)}{\varphi}$ is decreasing function.
	\item \label{B_c/B_s}$\frac{B_c}{B_s}$ is decreasing function.
	\item \label{B_s/sin}$\frac{B_s}{\sin{(\varphi)}}$ is increasing function.
	\end{enumerate}
\begin{proof}
	Find the derivative of the function $ \frac {c (\varphi)} {\varphi} $ and show that it is negative.
	$$\left(\frac{c(\varphi)}{\varphi}\right)'=\dfrac{c'(\varphi)\varphi-c(\varphi)}{\varphi^2}.$$
	In order to prove that this derivative is negative, it is sufficient to show that $c'(\varphi)\varphi-c(\varphi)<0$.
	$$(c'(\varphi)\varphi-c(\varphi))'=c''(\varphi)<0$$
	From which it follows that
	$c'(\varphi)\varphi-c(\varphi)<c(0)=0$, this means that the statement (\ref{a/phi}) is true. The statement (\ref{b/phi}) is proved similarly.\\
	Since $\frac{B_c}{B_s}>0$, decreasing $\frac{B_c}{B_s}$ is equivalent to decreasing $\frac{B_c^4}{B_s^4}$. By taking the derivative of the function $\frac{B_c^4}{B_s^4}$ and expand into factors we get that:
	$$\left(\frac{B_c^4}{B_s^4}\right)'=-\dfrac{2\sin{(\varphi)}(3-3\cos{(\varphi)}+\cos^2{(\varphi)})(2-\cos{(\varphi)})(3+\cos{(\varphi)})}{(1-\cos{(\varphi)})^3(7-4\cos{(\varphi)}+\cos^2{(\varphi)})^2} < 0.$$
	This means that the function $\frac{B_c}{B_s}$ decreases.\\
	Since $\frac{B_s}{\sin{(\varphi)}}>0$ increasing $\frac{B_s}{\sin{(\varphi)}}$ is equivalent to decreasing  $\frac{B_s^4}{\sin^4{(\varphi)}}$. By taking the derivative of the function $\frac{B_s^4}{\sin^4{(\varphi)}}$ and expand into factors we get that:
	$$\left(\frac{B_s^4}{\sin^4{(\varphi)}}\right)'=\dfrac{6(3-\cos{(\varphi)})(1-\cos{(\varphi)})^3}{\sin^5{(\varphi)}} > 0.$$
	This means that the function $\frac{B_s}{\sin(\varphi)}$ increases.
\end{proof}
\item \begin{enumerate}
	\item \label{a v 0} If $\varphi$ is sufficiently small then $c=\frac{1}{2}\varphi-\frac{3}{48}\varphi^3+o(\varphi^3)$.
	\item \label{b v 0} If $\varphi$ is sufficiently small then $b=\frac{\sqrt{3}}{2}\varphi-\frac{\sqrt{3}}{48}\varphi^3+o(\varphi^3)$.
	\item \label{B v 0} If $\varphi$ is sufficiently small then $B=\varphi-\frac{1}{48}\varphi^3+o(\varphi^3)$.
	\item \label{C v 0} If $\varphi$ is sufficiently small then $C=\frac{\sqrt{3}}{16}\varphi^3+o(\varphi^3)$.
\end{enumerate}
	\begin{proof}
	 Let's find $\sin{\beta}$ near the point $\varphi=0$, with an accuracy of  $o(\varphi)$. From the formulas (\ref{opred B_s}) and  (\ref{opred psi_s})  it follows:
	$$
	B_s=\sqrt[4]{(1-1+\frac{\varphi^2}{2}+o(\varphi^2))^2(7-4+1+o(\varphi))}=\varphi+o(\varphi)
	$$
	$$
	\psi_s=\frac{\pi-\arctan{\frac{\sqrt{3}(3-1+o(\varphi))}{2+o(\varphi)}}}{2}=\frac{\pi-\arctan{\sqrt{3}}+o(\varphi)}{2}=\frac{\pi}{3}+o(\varphi)
	$$
	From which it follows that near the point $\varphi=0$
	\begin{equation}
	\sin{\beta}=\varphi e^{\frac{\pi i}{3}} + o(\varphi)
	\end{equation}
	Then from the formulas (\ref{beta'}) and (\ref{beta''}) it is not difficult to get:
	\begin{equation}
	\beta'(0)=e^{\frac{\pi i}{3}}
	\label{b' v 0}
	\end{equation}
	\begin{equation}
	\beta''(0)=0
	\label{b'' v 0}
	\end{equation}
	Now we'll find $\beta'''(0)$
	$$
	\begin{aligned}
	\beta'''(0)&=\lim_{\varphi \to 0}{\dfrac{\sqrt{3}(2(1-\cos{(\varphi)})\sin{(\varphi)} \sin^3{(\beta(\varphi))}-3(1-\cos{(\varphi)})^2\sin^2{(\beta(\varphi))}\cos{(\beta(\varphi))}\beta'(\varphi))}{\sin^6{(\beta(\varphi))}}}e^{\frac{\pi i}{6}}\\
	&=\lim_{\varphi \to 0}{\dfrac{\sqrt{3}(\varphi^6e^{\pi i}-\frac{3}{4}\varphi^6e^{\pi i})}{\varphi^6}}e^{\frac{\pi i}{6}}=\frac{\sqrt{3}}{4}e^{\frac{ 7\pi i}{6}}
	\end{aligned}
	$$
	We get that:
	\begin{equation}
	\beta'''(0)=\frac{\sqrt{3}}{4}e^{\frac{ 7\pi i}{6}}
	\label{b''' v 0}
	\end{equation}
	Since $ \beta(0)=0$, and taking into account (\ref{b' v 0}), (\ref{b'' v 0}), and (\ref{b''' v 0}) near the point $\varphi=0$, we have:
	\begin{equation}
	\beta(\varphi)=\varphi\left(\frac{1}{2}+\frac{\sqrt{3}i}{2}\right)-\varphi^3\left(\frac{3}{48}+\frac{\sqrt{3}i}{48}\right)+o(\varphi^3)
	\label{beta v 0}
	\end{equation}
	Hence the statements (\ref{a v 0}) and (\ref{b v 0}) is true.\\
	Let's introduce the function $D(\varphi)=\sin{(\beta)}e^{\frac{-\pi i}{3}}$ then $B=\Re{(D)}$,  $C=-\Im{(D)}$. Let's find the value of the first three derivatives of the function $D (\varphi)$ at the point $ \varphi=0$.
	$$
	D'(\varphi)=\cos{(\beta)}\beta'e^{\frac{-\pi i}{3}}
	$$
	Taking into account the formula (\ref{b' v 0})
	\begin{equation}
	D'(0)=1
	\label{D' v 0}
	\end{equation}
	$$
	D''(\varphi)=(\cos{(\beta)}\beta''-\sin{(\beta)}(\beta')^2)e^{\frac{-\pi i}{3}}
	$$
	Given that $ \beta(0)=0$, and the formulas (\ref{b'' v 0}), we get
	\begin{equation}
	D''(0)=0
	\label{D'' v 0}
	\end{equation}
	$$
	D'''(\varphi)=(\cos{(\beta)}\beta'''-3\sin{(\beta)}\beta'\beta''-\cos{(\beta)}(\beta')^3)e^{\frac{-\pi i}{3}}
	$$
	Using the formulas (\ref{b' v 0}), (\ref{b'' v 0}), and (\ref{b''' v 0}), it is not difficult to get that:
	\begin{equation}
	D'''(0)=\frac{1}{8}-\frac{3\sqrt{3}i}{8}
	\label{D''' v 0}
	\end{equation}
	Then, taking into account (\ref{D' v 0}), (\ref{D'' v 0}) and (\ref{D''' v 0}) near the point $ \varphi=0$, we have:
	\begin{equation}
	D(\varphi)=\varphi+\varphi^3\left(\frac{1}{48}-\frac{\sqrt{3}i}{16}\right)+o(\varphi^3)
	\label{D v 0}
	\end{equation}
		Hence the statements (\ref{B v 0}) and (\ref{C v 0}) is true.\\	
	\end{proof}
\end{enumerate}
\section{Chebyshev polynomial}
\label{Razdel pol Cheb}
 To solve this problem, we need to solve the equation $\det{ T_n(a-g(\varphi))}=0 $, $\varphi \in (0,\pi) $. To find the determinant we will use the results obtained in the paper \cite{Elou}.
Let's define Chebyshev polynomials $\{Q_n\}$, $\{U_n\}$, $\{V_n\}$, $\{W_n\}$, which satisfy the same recurrent formula
\[Q_{n+1}(x)=2xQ_{n}(x)-Q_{n-1}(x),\;  n=1,2,\dots\]
and the different initial conditions are:
\[\begin{aligned}
Q_{0}(x)=U_{0}(x)=1,& & & 2Q_{1}(x)=U_{1}(x)=2x\\
W_{0}(x)=V_{0}(x)=1,& & & W_{1}(x)=V_{1}(x)+2=2x+1
\end{aligned} \]
It is easy to check that these polynomials satisfy the following conditions
\begin{equation}
\begin{aligned}
Q_{n}(\cos{\theta})=\cos{n\theta},& & & U_{n}(\cos{\theta})=\frac{\sin{(n+1)\theta}}{\sin{\theta}} \\
V_{n}(\cos{\theta})=\frac{\cos{(n+\frac{1}{2})\theta}}{\cos{\frac{\theta}{2}}},& & & U_{n}(\cos{\theta})=\frac{\sin{(n+\frac{1}{2})\theta}}{\sin{\frac{1}{2}\theta}} \\
\end{aligned}
\label{Svoistva pol Cheb}
\end{equation}
In \cite{Elou}, for the generating polynomial $a(t)= \sum\limits_{k=-r}^{r}a_kt^k $, where $a_r\ne 0 $, $a_k=a_{-k}$, the following theorem was proved.\\
\begin{theorem}[{\cite{Elou} Theorem 1}]
	\label{Theorem Pol Cheb}
	Let $\xi_j $ and $\frac{1}{\xi_j}$ be the (distinct) zeros of the polynomial $g_1(t)=t^ra(t)$. Then, for all $p\geq 1$ $\det{T_{2p}}$ equals
	\[\frac{a_{r}^{2p}}{2^{r(r-1)}}\times
	\frac{\begin{vmatrix}
		V_p(\alpha_1)&\dots & V_p(\alpha_r)\\
		\vdots & \ddots & \vdots\\
		V_{p+r-1}(\alpha_1)&\cdots & V_{p+r-1}(\alpha_r)
		\end{vmatrix}}{\prod\limits_{1 \leq i\leq j \leq r}{(\alpha_j-\alpha_i)}}\times
	\frac{\begin{vmatrix}
		W_p(\alpha_1)&\dots & W_p(\alpha_r)\\
		\vdots & \ddots & \vdots\\
		W_{p+r-1}(\alpha_1)&\cdots & W_{p+r-1}(\alpha_r)
		\end{vmatrix}}{\prod\limits_{1 \leq i\leq j \leq r}{(\alpha_j-\alpha_i)}}
	\]
	and $\det{T_{2p+1}}$ equals
	\[\frac{(-1)^{r}a_{r}^{2p}}{2^{r(r-2)}}\times
	\frac{\begin{vmatrix}
		U_p(\alpha_1)&\dots & U_p(\alpha_r)\\
		\vdots & \ddots & \vdots\\
		U_{p+r-1}(\alpha_1)&\cdots & U_{p+r-1}(\alpha_r)
		\end{vmatrix}}{\prod\limits_{1 \leq i\leq j \leq r}{(\alpha_j-\alpha_i)}}\times
	\frac{\begin{vmatrix}
		Q_{p+1}(\alpha_1)&\dots & Q_{p+1}(\alpha_r)\\
		\vdots & \ddots & \vdots\\
		Q_{p+r-1}(\alpha_1)&\cdots & Q_{p+r-1}(\alpha_r)
		\end{vmatrix}}{\prod\limits_{1 \leq i\leq j \leq r}{(\alpha_j-\alpha_i)}}
	,\]
	where $\alpha_k=\frac{1}{2}(\xi_k+\frac{1}{\xi_k})$ $(k=1,\dots,r)$ are the zeros of the polynomial $h_1(x)=a_0+2\sum\limits_{k=1}^{r}{a_k}Q_k(x)$.
\end{theorem}
In our case $g_1(t)=(t^2-2t+1)^{3}-\lambda t^3$, taking into account that $\lambda = g(\varphi)=(2\cos{\varphi}-2)^3$ it is easy to get that:
$$
\begin{aligned}
&\alpha_1=\cos{\varphi}\\
&\alpha_2=1+(\cos{\varphi}-1)e^{\frac{2\pi i}{3}}\\
&\alpha_3=1+(\cos{\varphi}-1)e^{\frac{-2\pi i}{3}}
\end{aligned}
$$
Next, we show that there are such regular functions $ \beta=\beta(\varphi) $ and $\gamma =\gamma (\varphi)$ , $ \varphi \in (-\pi,\pi]$ which will satisfy the equations:
\begin{equation}
\begin{aligned}
&\cos{\beta}=1+(\cos{\varphi}-1)e^{\frac{2\pi i}{3}}=\alpha_2\\
&\cos{\gamma}=1+(\cos{\varphi}-1)e^{\frac{-2\pi i}{3}}=\alpha_3
\label{ur opred beta i gamma}
\end{aligned}
\end{equation}
To do this, it is enough to show that each of the multifunctions $E_2=-i\Log(\alpha_2+i(1-\alpha_2^2)^\frac{1}{2})$ and $E_3=-i\Log(\alpha_3+i(1 - \alpha_3^2)^\frac{1}{2})$ has at least one regular branch for $ \varphi \in (-\pi,\pi]$. Given the notation (\ref{opred B_c})-(\ref{opred psi_s}), it is not difficult to make sure that
 $$
 \alpha_3=B_c e^{-i\psi_c}=\overline{\alpha_2}
 $$
 And also
 $$(1-\alpha_3^2)^\frac{1}{2}=B_s e^{-i\psi_s}$$
 Note that in this case, we choose one of the two regular branches. With this in mind, it is
 sufficient to show that each of the functions $\tilde E_2=-i\Log(B_c e^{i\psi_c}+iB_s e^{i\psi_s})$ and $\tilde E_3=-i\Log(B_c e^{-i\psi_c}+iB_s e^{-i\psi_s})$, has at least one regular branch for $\varphi \in (-\pi,\pi]$. Note also that $B_c$, $\psi_c $, $B_s$, and $\psi_s$ are even functions.
\begin{lm}
	Let $ \varphi\in (-\pi,\pi]$ then the multifunctions $$ \tilde E_2=-i\Log(B_c e^{i\psi_c}+iB_s e^{i\psi_s})$$ and $$ \tilde E_3=-i\Log(B_c e^{-i\psi_c}+iB_s e^{-i\psi_s})$$ have regular branches $\beta (\varphi)$ and $\gamma(\varphi)$ respectively.
\end{lm}
\begin{proof} In order for the multifunction $ \tilde E_2$ to have a regular branch, it is sufficient that the curve $B_c e^{i\psi_c}+iB_s e^{i\psi_s}$ (which is smooth contour) lies inside a simply connected region that does not contain the point zero. In the item \ref{Bc_cos-Bs_sin} of the section \ref{Vspomogatelnie rezultati}, it was shown that $\Re({B_c e^{i\psi_c}+iB_s e^{i\psi_s}})>0$, with $ \varphi \in (0,\pi)$, since $B_c$, $\psi_c $, $B_s$, and $\psi_s$ are even functions, and $B_c(0)\cos{(\psi_c(0))}=1$, $B_s e^{i\psi_s}=0$ then $\Re{(\cos{\beta}+i\sin{\beta})}>0$, when $ \varphi \in (- \pi,\pi]$. This means that the smooth contour $B_c e^{i\psi_c}+iB_s e^{i\psi_s}$ lies in a simply connected region that does not contain the point zero, and therefore the multifunction $ \tilde E_2$ have a regular branch, for $ \varphi \in (-\pi,\pi]$. The function $ \Log$ has an infinite number of regular branches, but to choose one of them, it is enough to determine its value at one point, let's put $ \beta (0)=0$. It is not difficult to make sure that this value meets the conditions (\ref{ur opred beta i gamma}). Similarly, the multifunction $ \tilde E_3$ have a regular branch, when $\varphi \in (-\pi,\pi]$ with the value $\gamma (0)=0$.
\end{proof}
So we proved that there are regular functions $ \beta(\alpha)$ and $\gamma(\varphi)$ which satisfy the equalities (\ref{ur opred beta i gamma}), and
$$
\begin{aligned}
\cos{(\beta)}=B_c e^{i \psi_c}\\
\sin{(\beta)}=B_s e^{i \psi_s}\\
\cos{(\gamma)}=B_c e^{-i \psi_c}\\
\sin{(\gamma)}=B_s e^{-i \psi_s}
\end{aligned}
$$
\section{Proof of main the results}
\begin{proof}[Proof of theorem \ref{Gl_theor}]
	If $n=2p$ then by theorem \ref{Theorem Pol Cheb} we have
	\begin{equation}
	\begin{aligned}
	det{(T_{2p}(a-g(\varphi)))}=\frac{1}{2^6}\times
	\frac{\begin{vmatrix}
		V_p(\cos{\varphi})&V_p(\cos{\beta}) & V_p(\cos{\gamma})\\
		V_{p+1}(\cos{\varphi}) & V_{p+1}(\cos{\beta}) & V_{p+1}(\cos{\gamma})\\
		V_{p+2}(\cos{\varphi})&V_{p+2}(\cos{\beta}) & V_{p+2}(\cos{\gamma})
		\end{vmatrix}}{(\cos{\gamma}-\cos{\beta})(\cos{\gamma}-\cos{\varphi})(\cos{\beta}-\cos{\varphi})}\\
	\times
	\frac{\begin{vmatrix}
		W_p(\cos{\varphi})&W_p(\cos{\beta}) & W_p(\cos{\gamma})\\
		W_{p+1}(\cos{\varphi}) & W_{p+1}(\cos{\beta}) & W_{p+1}(\cos{\gamma})\\
		W_{p+2}(\cos{\varphi})&W_{p+2}(\cos{\beta}) & W_{p+2}(\cos{\gamma})
		\end{vmatrix}}{(\cos{\gamma}-\cos{\beta})(\cos{\gamma}-\cos{\varphi})(\cos{\beta}-\cos{\varphi})}
	\end{aligned}
	\label{obshai formula T2p}
	\end{equation}
	It is easy to check that for $ \varphi \in (0, \pi)$ $\cos{\varphi}$, $\cos{\gamma}$, $\cos{\beta}$ are pairwise distinct, which means that the equation $\det{(T_{2p}(a-g(\varphi)))}=0$ is equivalent to the equation:
	\begin{equation}
	\begin{aligned}
	\begin{vmatrix}
	V_p(\cos{\varphi})&V_p(\cos{\beta}) & V_p(\cos{\gamma})\\
	V_{p+1}(\cos{\varphi}) & V_{p+1}(\cos{\beta}) & V_{p+1}(\cos{\gamma})\\
	V_{p+2}(\cos{\varphi})&V_{p+2}(\cos{\beta}) & V_{p+2}(\cos{\gamma})
	\end{vmatrix}\\
	\times
	\begin{vmatrix}
	W_p(\cos{\varphi})&W_p(\cos{\beta}) & W_p(\cos{\gamma})\\
	W_{p+1}(\cos{\varphi}) & W_{p+1}(\cos{\beta}) & W_{p+1}(\cos{\gamma})\\
	W_{p+2}(\cos{\varphi})&W_{p+2}(\cos{\beta}) & W_{p+2}(\cos{\gamma})
	\end{vmatrix}=0
	\end{aligned}
	\label{uravnenie dlya T2p}
	\end{equation}
	Taking into account the properties (\ref{Svoistva pol Cheb}), the equation (\ref{uravnenie dlya T2p}) will take the form:
	\begin{equation}
	\begin{aligned}
	\begin{vmatrix}
	\dfrac{\cos{((p+\frac{1}{2})\varphi)}}{\cos{\frac{\varphi}{2}}}&\dfrac{\cos{((p+\frac{1}{2})\beta)}}{\cos{\frac{\beta}{2}}} & \dfrac{\cos{((p+\frac{1}{2})\gamma)}}{\cos{\frac{\gamma}{2}}}\\
	\dfrac{\cos{((p+\frac{3}{2})\varphi)}}{\cos{\frac{\varphi}{2}}} &\dfrac{\cos{((p+\frac{3}{2})\beta)}}{\cos{\frac{\beta}{2}}} & \dfrac{\cos{((p+\frac{3}{2})\gamma)}}{\cos{\frac{\gamma}{2}}}\\
	\dfrac{\cos{((p+\frac{5}{2})\varphi)}}{\cos{\frac{\varphi}{2}}}&\dfrac{\cos{((p+\frac{5}{2})\beta)}}{\cos{\frac{\beta}{2}}} & \dfrac{\cos{((p+\frac{5}{2})\gamma)}}{\cos{\frac{\gamma}{2}}}
	\end{vmatrix}\\
	\times
	\begin{vmatrix}
	\dfrac{\sin{((p+\frac{1}{2})\varphi)}}{\sin{\frac{\varphi}{2}}}&\dfrac{\sin{((p+\frac{1}{2})\beta)}}{\sin{\frac{\beta}{2}}} & \dfrac{\sin{((p+\frac{1}{2})\gamma)}}{\sin{\frac{\gamma}{2}}}\\
	\dfrac{\sin{((p+\frac{3}{2})\varphi)}}{\sin{\frac{\varphi}{2}}} &\dfrac{\sin{((p+\frac{3}{2})\beta)}}{\sin{\frac{\beta}{2}}} & \dfrac{\sin{((p+\frac{3}{2})\gamma)}}{\sin{\frac{\gamma}{2}}}\\
	\dfrac{\sin{((p+\frac{5}{2})\varphi)}}{\sin{\frac{\varphi}{2}}}&\dfrac{\sin{((p+\frac{5}{2})\beta)}}{\sin{\frac{\beta}{2}}} & \dfrac{\sin{((p+\frac{5}{2})\gamma)}}{\sin{\frac{\gamma}{2}}}
	\end{vmatrix}=0
	\end{aligned}
	\label{uravnenie2 dlya T2p}
	\end{equation}
	It is not difficult to check that $\sin(\varphi)\neq 0$, $\sin(\beta)\neq 0$ and $\sin(\gamma)\neq 0$ if $ \varphi \in (0,\pi)$. Then, since $n=2p$, the set of solutions to the equation (\ref{uravnenie2 dlya T2p}) coincides with the union of the sets of solutions to the equations:
	\begin{equation}
	\begin{vmatrix}
	\cos{(\frac{n+1}{2}\varphi)}&\cos{(\frac{n+1}{2}\beta)} &\cos{(\frac{n+1}{2}\gamma)}\\
	\cos{(\frac{n+3}{2}\varphi)}&\cos{(\frac{n+3}{2}\beta)} &\cos{(\frac{n+3}{2}\gamma)}\\
	\cos{(\frac{n+5}{2}\varphi)}&\cos{(\frac{n+5}{2}\beta)} &\cos{(\frac{n+5}{2}\gamma)}
	\end{vmatrix}=0\\
	\label{Pervoe uravnenie dlya T2p}
	\end{equation}
	and
	\begin{equation}
	\begin{vmatrix}
	\sin{(\frac{n+1}{2}\varphi)}&\sin{(\frac{n+1}{2}\beta)} &\sin{(\frac{n+1}{2}\gamma)}\\
	\sin{(\frac{n+3}{2}\varphi)}&\sin{(\frac{n+3}{2}\beta)} &\sin{(\frac{n+3}{2}\gamma)}\\
	\sin{(\frac{n+5}{2}\varphi)}&\sin{(\frac{n+5}{2}\beta)} &\sin{(\frac{n+5}{2}\gamma)}
	\end{vmatrix}=0\\
	\label{vtoroe uravnenie dlya T2p}
	\end{equation}
	If $n=2p+1$, similar reasoning will lead to the same equations (\ref{Pervoe uravnenie dlya T2p}) and (\ref{vtoroe uravnenie dlya T2p}).
	Expanding the determinant in the formula (\ref{Pervoe uravnenie dlya T2p}) by the first column, we obtain the following equation:
	\begin{equation}
	\cos{\left(\frac{n+1}{2}\varphi \right)}S_1-\cos{\left(\frac{n+3}{2}\varphi \right)}S_2+\cos{\left(\frac{n+5}{2}\varphi \right)}S_3=0,
	\label{uravnenie posle rask det}
	\end{equation}
	where $$S_1=\cos{\left(\frac{n+3}{2}\beta \right)}\cos{\left(\frac{n+5}{2}\gamma \right)}-\cos{\left(\frac{n+5}{2}\beta \right)}\cos{\left(\frac{n+3}{2}\gamma \right)},$$\\
	$$S_2=\cos{\left(\frac{n+1}{2}\beta \right)}\cos{\left(\frac{n+5}{2}\gamma \right)}-\cos{\left(\frac{n+5}{2}\beta \right)}\cos{\left(\frac{n+1}{2}\gamma \right)},$$\\
	$$S_3=\cos{\left(\frac{n+1}{2}\beta \right)}\cos{\left(\frac{n+3}{2}\gamma \right)}-\cos{\left(\frac{n+3}{2}\beta \right)}\cos{\left(\frac{n+1}{2}\gamma \right)}.$$\\
	Consider the sum of the first and third terms, using the fact that:
	$$\cos{\left(\frac{n+1}{2}\varphi \right)}=\cos{\left(\frac{n+3}{2}\varphi-\varphi \right)}=\cos{\left(\frac{n+3}{2}\varphi \right)}\cos{\left(\varphi\right)}+\sin{\left(\frac{n+3}{2}\varphi \right)}\sin{\left(\varphi\right)}$$
	and
	$$\cos{\left(\frac{n+5}{2}\varphi \right)}=\cos{\left(\frac{n+3}{2}\varphi+\varphi \right)}=\cos{\left(\frac{n+3}{2}\varphi \right)}\cos{\left(\varphi\right)}-\sin{\left(\frac{n+3}{2}\varphi \right)}\sin{\left(\varphi\right)}$$
	then
	\begin{equation}
	\begin{aligned}
	\cos{\left(\frac{n+1}{2}\varphi \right)}S_1&+\cos{\left(\frac{n+5}{2}\varphi \right)}S_3\\
	=\cos{\left(\frac{n+3}{2}\varphi \right)}\cos{\left(\varphi\right)}(S_1+S_3)&+\sin{\left(\frac{n+3}{2}\varphi \right)}\sin{\left(\varphi\right)}(S_1-S_3)
	\end{aligned}
	\label{cosS_1+cosS_3}
	\end{equation}

	Let's find $S_1+S_3$ by grouping the first term of $S_1$ with the second of $S_3$ and vice versa:
	$$
	\begin{aligned}
	S_1+S_3&=\cos{\left(\frac{n+3}{2}\beta \right)}\left(\cos{\left(\frac{n+5}{2}\gamma\right)}-\cos{\left(\frac{n+1}{2}\gamma \right)} \right)\\
	&-\cos{\left(\frac{n+3}{2}\gamma \right)}\left(\cos{\left(\frac{n+5}{2}\beta\right)}-\cos{\left(\frac{n+1}{2}\beta \right)} \right)
	\end{aligned}
	$$
	Using the cosine difference formula, it is not difficult to obtain
\begin{equation}
	S_1+S_3=2\cos{\left(\frac{n+3}{2}\gamma \right)}\sin{\left(\frac{n+3}{2}\beta \right)}\sin{\left(\beta \right)}-2\cos{\left(\frac{n+3}{2}\beta \right)}\sin{\left(\frac{n+3}{2}\gamma \right)}\sin{\left(\gamma \right)}
	\label{S_1+S_3}
\end{equation}
	From similar reasoning we get:
	$$
	S_1-S_3=2\cos{\left(\frac{n+3}{2}\gamma \right)}\cos{\left(\frac{n+3}{2}\beta \right)}\cos{\left(\gamma \right)}-2\cos{\left(\frac{n+3}{2}\beta \right)}\cos{\left(\frac{n+3}{2}\gamma \right)}\cos{\left(\beta \right)}
	$$
	If we take the total multiplier out of the brackets:
	\begin{equation}
		S_1-S_3=2\cos{\left(\frac{n+3}{2}\gamma \right)}\cos{\left(\frac{n+3}{2}\beta \right)}(\cos{\left(\gamma \right)}-\cos{\left(\beta \right)})
		\label{S_1-S_3}
	\end{equation}
	Consider $S_2$ by writing it as:
	 $$S_2=\cos{\left(\frac{n+3}{2}\beta-\beta \right)}\cos{\left(\frac{n+3}{2}\gamma+\gamma \right)}-\cos{\left(\frac{n+3}{2}\beta+\beta \right)}\cos{\left(\frac{n+3}{2}\gamma-\gamma \right)}$$	
	Using the formula of the cosine of the sum, expanding the brackets and giving similar terms, we get the following:	
\begin{equation}
	S_2=2\cos{\left(\frac{n+3}{2}\gamma \right)}\sin{\left(\frac{n+3}{2}\beta \right)}\cos{\left(\gamma \right)}\sin{\left(\beta \right)}-2\cos{\left(\frac{n+3}{2}\beta \right)}\sin{\left(\frac{n+3}{2}\gamma \right)}\cos{\left(\beta \right)}\sin{\left(\gamma \right)}
	\label{cosS_2}
\end{equation}
	Substituting the equalities (\ref{cosS_1+cosS_3})-(\ref{cosS_2}) into the equation (\ref{uravnenie posle rask det}), we get the following equation:
\begin{equation}
	\begin{aligned}
	2\cos{\left(\frac{n+3}{2}\varphi \right)}\cos{\left(\frac{n+3}{2}\gamma \right)}\sin{\left(\frac{n+3}{2}\beta \right)}\sin{\left(\beta \right)}(\cos{\left(\varphi \right)}-\cos{\left(\gamma \right)})\\
	-	2\cos{\left(\frac{n+3}{2}\varphi \right)}\cos{\left(\frac{n+3}{2}\beta \right)}\sin{\left(\frac{n+3}{2}\gamma \right)}\sin{\left(\gamma \right)}(\cos{\left(\varphi \right)}-\cos{\left(\beta \right)})\\
	=2\sin{\left(\frac{n+3}{2}\varphi \right)}\sin{\left(\varphi\right)}\cos{\left(\frac{n+3}{2}\gamma \right)}\cos{\left(\frac{n+3}{2}\beta \right)}(\cos{\left(\beta \right)}-\cos{\left(\gamma \right)})
	\end{aligned}
	\label{uravnenie 111}
\end{equation}
	It is not difficult to make sure that
	$$
	\begin{aligned}
	&\cos{\left(\beta \right)}-\cos{\left(\gamma \right)}=(\cos{\left(\varphi \right)}-1)\sqrt{3}i\\
	&\cos{\left(\varphi \right)}-\cos{\left(\beta \right)}=(\cos{\left(\varphi \right)}-1)\left(\dfrac{3}{2} -\dfrac{\sqrt{3}i}{2}\right)\\
	&\cos{\left(\varphi \right)}-\cos{\left(\gamma \right)}=(\cos{\left(\varphi \right)}-1)\left(\dfrac{3}{2} +\dfrac{\sqrt{3}i}{2}\right)\\
	\end{aligned}
	$$
	So:
	\begin{equation}
	\begin{aligned}
	&\dfrac{\cos{\left(\varphi \right)}-\cos{\left(\beta \right)}}{\cos{\left(\beta \right)}-\cos{\left(\gamma \right)}}=-\dfrac{1}{2}-\dfrac{\sqrt{3}i}{2}=-e^{\frac{\pi i}{3}}\\
	&\dfrac{\cos{\left(\varphi \right)}-\cos{\left(\gamma \right)}}{\cos{\left(\beta \right)}-\cos{\left(\gamma \right)}}=\dfrac{1}{2}-\dfrac{\sqrt{3}i}{2}=-e^{\frac{2\pi i}{3}}
	\end{aligned}
	\label{delenie}
	\end{equation}
	Dividing the right and left sides of the equation (\ref{uravnenie 111}) by $\cos{\left(\gamma \right)}-\cos{\left(\beta \right)}$, using the equalities (\ref{delenie}), and also taking into account that $\sin{\left (\beta \right)}=B_se^{i\psi_s} $ and $\sin{\left(\gamma \right)}=B_se^{-i\psi_s} $ we get the equation:
	\begin{equation}
	\begin{aligned}
	&2\cos{\left(\frac{n+3}{2}\varphi \right)}\cos{\left(\frac{n+3}{2}\beta \right)}\sin{\left(\frac{n+3}{2}\gamma \right)}B_se^{-i\psi_s+\frac{\pi i}{3}} \\
	&-2\cos{\left(\frac{n+3}{2}\varphi \right)}\cos{\left(\frac{n+3}{2}\gamma \right)}\sin{\left(\frac{n+3}{2}\beta \right)}B_se^{i\psi_s+\frac{2\pi i}{3}}\\
	&=2\sin{\left(\frac{n+3}{2}\varphi \right)}\sin{\left(\varphi\right)}\cos{\left(\frac{n+3}{2}\gamma \right)}\cos{\left(\frac{n+3}{2}\beta \right)}
	\end{aligned}
	\label{uravnenie 222}
	\end{equation}
	Given that $ \beta=c+ib$, $\gamma=c-ib$, we get
\begin{equation}
\begin{aligned}
	2\cos{\left(\frac{n+3}{2}\beta \right)}\sin{\left(\frac{n+3}{2}\gamma \right)}=\sin{((n+3)c)}-i\sinh{((n+3)b)}\\
	2\cos{\left(\frac{n+3}{2}\gamma \right)}\sin{\left(\frac{n+3}{2}\beta \right)}=\sin{((n+3)c)}+i\sinh{((n+3)b)}\\
	2\cos{\left(\frac{n+3}{2}\beta \right)}\cos{\left(\frac{n+3}{2}\gamma \right)}=\cos{((n+3)c)}+\cosh{((n+3)b)}
\end{aligned}
	\label{Perehod k a i b}
\end{equation}
Denote by $S_5$ the left side of the equation (\ref{uravnenie 222}) without the multiplier $ \cos{\left(\frac{n+3}{2}\varphi \right)}$, then taking into account (\ref{Perehod k a i b}) we get:
$$
\begin{aligned}
S_5&=(\sin{((n+3)c)}-i\sinh{((n+3)b)})B_s\left(\cos{\left(-\psi_s+\frac{\pi}{3}\right)}+i\sin{\left(-\psi_s+\frac{\pi}{3}\right)}\right)\\
&-(\sin{((n+3)c)}+i\sinh{((n+3)b)})B_s\left(\cos{\left(\psi_s+\frac{2 \pi}{3}\right)}+i\sin{\left(\psi_s+2\frac{\pi}{3}\right)}\right)
\end{aligned}
$$
Which is equivalent to:
\begin{equation}
\begin{aligned}
S_5&=\sin{((n+3)c)}B_s\left(\cos{\left(-\psi_s+\frac{\pi}{3}\right)}- \cos{\left(\psi_s+\frac{2 \pi}{3}\right)}\right)\\
&-i\sinh{((n+3)b)}B_s\left(\cos{\left(-\psi_s+\frac{\pi}{3}\right)}+ \cos{\left(\psi_s+\frac{2 \pi}{3}\right)}\right)\\
&+i\sin{((n+3)c)}B_s\left(\sin{\left(-\psi_s+\frac{\pi}{3}\right)}- \sin{\left(\psi_s+\frac{2 \pi}{3}\right)}\right)\\
&+\sinh{((n+3)b)}B_s\left(\sin{\left(-\psi_s+\frac{\pi}{3}\right)}+ \sin{\left(\psi_s+\frac{2 \pi}{3}\right)}\right)
\end{aligned}
\label{uravnenie 333}
\end{equation}
It is easy to make sure that:
\begin{equation}
\begin{aligned}
&B_s\left(\cos{\left(-\psi_s+\frac{\pi}{3}\right)}- \cos{\left(\psi_s+\frac{2 \pi}{3}\right)}\right)=2B\\
&B_s\left(\cos{\left(-\psi_s+\frac{\pi}{3}\right)}+ \cos{\left(\psi_s+\frac{2 \pi}{3}\right)}\right)=0\\
&B_s\left(\sin{\left(-\psi_s+\frac{\pi}{3}\right)}- \sin{\left(\psi_s+\frac{2 \pi}{3}\right)}\right)=0\\
&B_s\left(\sin{\left(-\psi_s+\frac{\pi}{3}\right)}+ \sin{\left(\psi_s+\frac{2 \pi}{3}\right)}\right)=2C.
\end{aligned}
\label{nahod B i C}
\end{equation}
Then, taking into account (\ref{uravnenie 333}), (\ref{nahod B i C}) and \ref{Perehod k a i b}, the equation (\ref{uravnenie 222}) will take the form
$$
\begin{aligned}
\cos{\left(\frac{n+3}{2}\varphi \right)}(2B\sin{((n+3)c)}+2C\sinh{((n+3)b)})\\
=\sin{\left(\frac{n+3}{2}\varphi \right)}\sin{\left(\varphi\right)}(\cos{((n+3)c)}+\cosh{((n+3)b)})
\end{aligned}
$$
Since $\sin{(\varphi)} \neq 0$, $b>0$ and hence $\cos{((n+3) c)}+\cos{((n+3)b)}>0$ when $\varphi \in (0, \pi)$, then the equation can be rewritten as
\begin{equation}
\sin{\left(\frac{n+3}{2}\varphi \right)}=\cos{\left(\frac{n+3}{2}\varphi \right)}2\dfrac{B\sin{((n+3)c)}+C\sinh{((n+3)b)})}{\sin{\left(\varphi\right)}(\cos{((n+3)c)}+\cosh{((n+3)b)})}
\label{uravnenie 444}
\end{equation}
Since $\sin{\left(\frac{n+3}{2}\varphi\right)}$ and $\cos{\left(\frac{n+3}{2}\varphi\right)}$ cannot be equal to zero at the same time, then the equation (\ref{uravnenie 444}) and therefore the equation (\ref{Pervoe uravnenie dlya T2p}) is equivalent to the equation (\ref{Glavnoe ur 1}). Doing similar calculations for the equation \ref{vtoroe uravnenie dlya T2p}, you can get that it is equivalent to the equation \ref{Glavnoe ur 2}.
\end{proof}
\begin{lm}
	\begin{enumerate}
		\item If $\varphi \in (\frac{\pi}{n+3},\pi) $ then
		\begin{equation}
		\frac{2}{n+3} \frac{f'(\varphi)}{1+f^2(\varphi)} < 0.8
		\label{333}
		\end{equation}
		\item If $\varphi \in (\frac{2 \pi}{n+3},\pi) $ then
		\begin{equation}
		\frac{2}{n+3} \frac{h'(\varphi)}{1+h^2(\varphi)} < 0.8
		\label{2333}
		\end{equation}
	\end{enumerate}
	\label{lemma ob ogran proizv}
\end{lm}
\begin{proof}
	\begin{equation}
	\frac{2}{n+3}f'(\varphi)=\frac{4}{n+3}(I_1+I_2+I_3+I_4),
	\label{f' ravno}
	\end{equation}
	where
	$$
	\begin{aligned}
	&I_1=\dfrac{B'\sin{((n+3)c)}+C'\sinh{((n+3)b)}}{\sin(\varphi)(\cos{((n+3)c)}+\cosh{((n+3)b)})}\\
	&I_2=\frac{(n+3)(B c'\cos{((n+3)c)}+C b'\cosh{((n+3)b)})}{\sin(\varphi)(\cos{((n+3)c)}+\cosh{((n+3)b)})}\\
	&I_3=\frac{(n+3)(B \sin{((n+3)c)}+C \sinh{((n+3)b)})(c'\sin{((n+3)c)}- b'\sinh{((n+3)b)})}{\sin(\varphi)(\cos{((n+3)c)}+\cosh{((n+3)b)})^2}\\
	&I_4=\frac{(B\sin{((n+3)c)}+C\sinh{((n+3)b)})\cos(\varphi)}{\sin^2(\varphi)(\cos{((n+3)c)}+\cosh{((n+3)b)})}.
	\end{aligned}
	$$
	Consider the sum $ \frac {4 (I_2 + I_3)} {n + 3} $, reduce to a common denominator, expand the brackets and give similar summands.
	$$
	\label{I_2+I_4}
	\begin{aligned}
	&\frac{4(I_2+I_3)}{n+3}= \frac{4(B c'\cos{((n+3)c)}+C b'\cosh{((n+3)b)})(\cos{((n+3)c)}+\cosh{((n+3)b)})}{\sin(\varphi)(\cos{((n+3)c)}+\cosh{((n+3)b)})^2}\\
	&+\frac{4(B \sin{((n+3)c)}+C \sinh{((n+3)b)})(c'\sin{((n+3)c)}- b'\sinh{((n+3)b)})}{\sin(\varphi)(\cos{((n+3)c)}+\cosh{((n+3)b)})^2}\\
	&=\frac{4(B a'\cos^2{((n+3)c)}+\cos{((n+3)c)}\cosh{((n+3)b)}(Bc'+Cb'))+Cb'\cosh^2{((n+3)b)})}{\sin(\varphi)(\cos{((n+3)c)}+\cosh{((n+3)b)})^2}\\
	&+\frac{4(B a'\sin^2{((n+3)c)}+\sin{((n+3)c)}\sinh{((n+3)b)}(Cc'-Bb'))-Cb'\sinh^2{((n+3)b)})}{\sin(\varphi)(\cos{((n+3)c)}+\cosh{((n+3)b)})^2}\\
	&=\frac{4(B c'+Cb' +\cos{((n+3)c)}\cosh{((n+3)b)}(Bc'+Cb'))+\sin{((n+3)c)}\sinh{((n+3)b)}(Cc'-Bb'))}{\sin(\varphi)(\cos{((n+3)c)}+\cosh{((n+3)b)})^2}
	\end{aligned}
	$$
	We will find $Ba'+Cb'$ and  $Ca'-Bb'$. Taking into account the formula (\ref{beta'}) and that $c'=\Re{(\beta')}$, $b'=\Im{(\beta')}$
	\begin{multline}
		Bc'+Cb'=\sin(\varphi)\left[\cos{\left(\frac{2 \pi}{3}-\psi_s\right)}\cos{\left(\psi_s-\frac{\pi}{3}\right)}-\sin{\left(\frac{2 \pi}{3}-\psi_s\right)}\sin{\left(\psi_s-\frac{\pi}{3}\right)}\right]=\sin(\varphi)\cos{\left(\frac{\pi}{3}\right)}
		\\=\frac{1}{2}\sin{(\varphi)}.
	\end{multline}\\
	Similarly $$Cc'-Bb'=-\frac{\sqrt{3}}{2}\sin{(\varphi)}.$$
	Then:\\
\begin{equation}
	\dfrac{4(I_2+I_3)}{n+3}=\dfrac{4(\frac{1}{2}+\frac{1}{2}\cos{((n+3)c)}\cosh{((n+3)b)}-\frac{\sqrt{3}}{2}\sin{((n+3)c)}\sinh{((n+3)b)})} {(\cos{((n+3)c)}+\cosh{((n+3)b)})^2}
	\label{I_2+I_3/n+3}
\end{equation}
	It is not difficult to check that
	$$B'(\varphi)= \Re{(\cot{(\beta)e^{\frac{\pi i}{3}}})}\sin{(\varphi)}$$  $$C'(\varphi)=- \Im{(\cot{(\beta)e^{\frac{\pi i}{3}}})}\sin{(\varphi)}.$$
	Then:
\begin{equation}
	I_1=\dfrac{|\cot{(\beta)}|[\cos{(\frac{\pi}{3}+\psi_c-\psi_s)}\sin{((n+3)c)}-\sin{(\frac{\pi}{3}+\psi_c-\psi_s)}\sinh{((n+3)b)}]}{(\cos{((n+3)c)}+\cosh{((n+3)b)})}.
	\label{I_1}
\end{equation}
	\begin{equation}
	I_4=\frac{|\sin{(\beta)}|\cos(\varphi)[\cos{(\psi_s-\frac{\pi}{3})}\sin{((n+3)c)}-\sin{(\psi_s-\frac{\pi}{3})}\sinh{((n+3)b)}]}{\sin^2{(\varphi)}(\cos{((n+3)c)}+\cosh{((n+3)b)})}
	\end{equation}
	Next, consider three cases $\varphi \in (\frac{\pi}{n+3}, \frac{2\pi}{n+3})$, $\varphi \in [\frac{2\pi}{n+3}, \frac{\pi}{2})$ and $\varphi \in [\frac{\pi}{2}, \pi)$.\\
	{\it Case 1.} $\varphi \in (\frac{\pi}{n+3}, \frac{2\pi}{n+3})$\\
	First, let's evaluate the value from below $\cosh{((n+3)b)}$. In (\ref{b}) of the section \ref{Vspomogatelnie rezultati}, it was shown that $b(\varphi)$ is an increasing function, so  $\cosh{((n+3)b)}<\cosh{((n+3)b(\frac{\pi}{n+3}))}$. From (\ref{b v 0}) of the section \ref{Vspomogatelnie rezultati} it follows that
	$$
	(n+3)b\left(\frac{\pi}{n+3}\right)=(n+3)\frac{\sqrt{3}}{2}\frac{\pi}{n+3}+o\left(\frac{1}{n+3}\right)=\frac{\sqrt{3}\pi}{2}+o\left(\frac{1}{n+3}\right)
	$$
	Then if $n$ is sufficiently large:
\begin{equation}
	(n+3)b\left(\frac{\pi}{n+3}\right)>2.72
	\label{ocenka b v pervom sluchae}
\end{equation}
	Where do we get the lower bound for  $\cosh{((n+3)b)}$
	\begin{equation}
	\cosh{((n+3)b)}>\cosh{(2.72)}>7.62
	\label{ocenka cosh v pervom sluchae}
	\end{equation}
	Consider $\frac{4(I_2+I_3)}{n+3}$, add and subtract in the numerator the summand $\frac{\sqrt{3}}{2}\sin{((n+3) c)}\cosh{((n+3)b)}$, then we get the following
	\begin{equation}
	\begin{aligned}
	\frac{4(I_2+I_3)}{n+3}&=\dfrac{4(\frac{1}{2}\cos{((n+3)c)}\cosh{((n+3)b)}-\frac{\sqrt{3}}{2}\sin{((n+3)c)}\cosh{((n+3)b)})}{(\cos{((n+3)c)}+\cosh{((n+3)b)})^2}\\
	&+\dfrac{4(\frac{1}{2}+\frac{\sqrt{3}}{2}\sin{((n+3)c)}(\cosh{((n+3)b)}-\sinh{((n+3)b)}))}{(\cos{((n+3)c)}+\cosh{((n+3)b)})^2}\\
	&=\dfrac{4(\cosh{((n+3)b)}\cos{((n+3)c+\frac{\pi}{3})}}{(\cos{((n+3)c)}+\cosh{((n+3)b)})^2}+\dfrac{4(\frac{1}{2}+\frac{\sqrt{3}}{2}\sin{((n+3)c)}e^{-(n+3)b})}{(\cos{((n+3)c)}+\cosh{((n+3)b)})^2}\\
	\end{aligned}
	\end{equation}
	Let's introduce the notation
	$$
	L_1=\dfrac{4(\cosh{((n+3)b)}\cos{((n+3)c+\frac{\pi}{3})}}{(\cos{((n+3)c)}+\cosh{((n+3)b)})^2},
	$$
	$$
	L_2=\dfrac{4(\frac{1}{2}+\frac{\sqrt{3}}{2}\sin{((n+3)c)}e^{-(n+3)b})}{(\cos{((n+3)c)}+\cosh{((n+3)b)})^2}
	$$
	From (\ref{a v 0}) of the section \ref{Vspomogatelnie rezultati} it follows that when $\varphi \in (\frac{\pi}{n+3}, \frac{2\pi}{n+3})$ $c(n+3)+\frac{\pi}{3}\in (\frac{5 \pi}{6}+o(\frac{1}{n}),\frac{8 \pi}{6}+o(\frac{1}{n}))$, this means that for a sufficiently large $n$ $\cos{((n+3)c+\frac{\pi}{3})}<0$, so $L_1<0$.\\
	Let's  estimate $|L_1|$
	$$
	|L_1|\leq \dfrac{4(\cosh{((n+3)b)}}{(\cos{((n+3)c)}+\cosh{((n+3)b)})^2}\leq\dfrac{4(\cosh{((n+3)b)}}{(-1+\cosh{((n+3)b)})^2}
	$$
	Taking into account (\ref{ocenka cosh v pervom sluchae}):
\begin{equation}
|L_1|<\dfrac{4\cdot 7.6}{(-1+7.6)^2}<0.7
\label{ocenka L_1}
\end{equation}
Let's  estimate $L_2$. From (\ref{a/phi}) and (\ref{a v 0}) it follows that $c<\frac{\varphi}{2}$, and since we consider the case when $ \varphi \in (\frac{\pi}{n+3}, \frac{2\pi}{n+3})$, then $c(n+3)<\pi$, which means $\sin{((n+3)c)}>0$, so both terms in $L_2$ greater than zero. It is not difficult to understand that with the growth of $b$, $L_2$ also grows, then taking into account (\ref{ocenka b v pervom sluchae}) and (\ref{ocenka cosh v pervom sluchae})
\begin{equation}
L_2<\dfrac{4(\frac{1}{2}+\frac{\sqrt{3}}{2}e^{-2.72})}{(-1+7.6)^2}<0.06
\label{ocenka L_2}
\end{equation}
Then taking into account (\ref{ocenka L_1}) and (\ref{ocenka L_2}):
\begin{equation}
-0.7<\frac{4(I_2+I_3)}{n+3}<0.06
\label{ocenka I_2+I_3 v pervom sluch}
\end{equation}
Let's  estimate $\frac{4I_1}{n+3}$. From (\ref{B_s}) and (\ref{B_c}) and the formulas (\ref{B_s'}) and (\ref{B_c'}) of the section \ref{Vspomogatelnie rezultati}, it is not difficult to get that in the neighborhood of the point $ \varphi = 0$ $B_c=1+o(\varphi)$, $B_s=\varphi + o(\varphi)$, then
\begin{equation}
|\cot{\beta}|=\frac{B_c}{B_s}=\frac{1}{\varphi} +o(1)
\label{cot v pervom sluch}
\end{equation}
 In (\ref{B_c/B_s}) of the section \ref{Vspomogatelnie rezultati}, it was shown that $| \cot {\beta}|$ is a decreasing function, so for a sufficiently large $n$ we have the following formula:
 \begin{equation}
 \frac{4|\cot{\beta}|}{n+3}<\frac{4}{n+3}\left(\frac{n+3}{\pi}+o(1)\right)<1.28
 \label{ocenka cor v pervom sluch}
 \end{equation}
From (\ref{psi_s}) and (\ref{psi_c}) of the section \ref{Vspomogatelnie rezultati}, and the formulas (\ref{psi_s'}) and (\ref{psi_c'}), it follows that near the point $ \varphi=0$ $\sin{(\frac{\pi}{3}+\psi_c-\psi_s)}=o(1)$. Since $(n+3)b<(n+3)\frac{\sqrt{3}\varphi}{2}< \sqrt{3} \pi$, then $ \sinh{((n+3)b)})$ is bounded. From which it follows that for a sufficiently large $n$:
 \begin{equation}
 |\sin{\left(\frac{\pi}{3}+\psi_c-\psi_s\right)}\sinh{((n+3)b)}|<0.01
 \label{ocenka sin v I_1 sluchae}
 \end{equation}
 Then substituting inequalities (\ref{ocenka cosh v pervom sluchae}), (\ref{ocenka cor v pervom sluch}) and (\ref{ocenka sin v I_1 sluchae}) into the formula (\ref{I_1}) we get that
 $$
 \frac{4I_1}{n+3}<\frac{1.28(1+0.01)}{-1+7.6}<0.2
 $$
It was shown above that in the case under consideration $\sin{((n+3)c)}>0$, and hence $$\cos{\left(\frac{\pi}{3}+\psi_c-\psi_s\right)}\sin{((n+3)c)}>0.$$ Taking into account \eqref{ocenka sin v I_1 sluchae} we get
 $$
 -0.01<\frac{4I_1}{n+3}
 $$
As a result
\begin{equation}
-0.01<\frac{4I_1}{n+3}<0.2
\label{ocenka I_1 v pervom sluchae}
\end{equation}
Let's  estimate  $\frac{4I_4}{n+3}$. Using the similar reasoning as in the estimation of $\frac{4 I_1}{n+3}$, it is easy to get that, near the point $ \varphi=0$:
$$
\frac{|\sin{(\beta)}|}{\sin{(\varphi)}}=1+o(1)<1.01
$$
$$
\frac{4}{n+3}\frac{1}{\sin{(\varphi)}}<\frac{4}{\pi}<1.28
$$
And also that
$$\sin{(\psi_s-\frac{\pi}{3})}=o(1)$$
Then
$$
\frac{4I_4}{n+3}<\frac{1.01*1.28(1+0.01)}{-1+7.6}<0.2
$$
And also
$$-0.01<\frac{4I_4}{n+3}$$
As a result
\begin{equation}
-0.01<\frac{4I_4}{n+3}<0.2
\label{ocenka I_4 v pervom sluchae}
\end{equation}
From the formula (\ref{f' ravno}) and the inequalities (\ref{ocenka I_2+I_3 v pervom sluch}), (\ref{ocenka I_1 v pervom sluchae}) and (\ref{ocenka I_4 v pervom sluchae})
we get that when $\varphi \in (\frac{\pi}{n+3}, \frac{2\pi}{n+3})$
$$\left|\frac{2}{n+3}f'\right|< 0.8$$
Which means
$$\left|\frac{2}{n+3}\frac{f'}{f^2+1}\right|< 0.8$$
{\it Case 2.} $\varphi \in [\frac{2 \pi}{n+3}, \frac{\pi}{2}]$. Using the similar reasoning as in the first case, it is not difficult to show that for a sufficiently large $n$
\begin{equation}
	\cosh{((n+3)b)}<115.
	\label{cosh vo vtorom sluch}
\end{equation}
Let's  estimate $\frac{4(I_2+I_3)}{n+3}$.
$$
\begin{aligned}
	\left|\dfrac{4(I_2+I_3)}{n+3}\right|&=	\left|\dfrac{4(\frac{1}{2}+\frac{1}{2}\cos{((n+3)c)}\cosh{((n+3)b)}-\frac{\sqrt{3}}{2}\sin{((n+3)c)}\sinh{((n+3)b)})} {(\cos{((n+3)c)}+\cosh{((n+3)b)})^2}\right|\\
&\leq\dfrac{4(\frac{1}{2}+|\frac{1}{2}\cos{((n+3)c)}\cosh{((n+3)b)}|+|\frac{\sqrt{3}}{2}\sin{((n+3)c)}\sinh{((n+3)b)})|} {(\cos{((n+3)c)}+\cosh{((n+3)b)})^2}\\
&< \dfrac{4(\frac{1}{2}+|\frac{1}{2}\cos{((n+3)c)}\cosh{((n+3)b)}|+|\frac{\sqrt{3}}{2}\sin{((n+3)c)}\cosh{((n+3)b)})|} {(\cos{((n+3)c)}+\cosh{((n+3)b)})^2}\\
&= \dfrac{4(\frac{1}{2}+\cosh{((n+3)b)}(|\frac{1}{2}\cos{((n+3)c)}|+|\frac{\sqrt{3}}{2}\sin{((n+3)c)}|))} {(\cos{((n+3)c)}+\cosh{((n+3)b)})^2}\\
&\leq \dfrac{4(\frac{1}{2}+\cosh{((n+3)b)})} {(\cos{((n+3)c)}+\cosh{((n+3)b)})^2}<\dfrac{4(\frac{1}{2}+115)} {(-1+115)^2}<0.04
\end{aligned}
$$
So
\begin{equation}
\left|\dfrac{4(I_2+I_3)}{n+3}\right|<0.04.
\label{I_2+I_3 vo vtorom sluch}
\end{equation}
Let's estimate $\frac{4I_1}{n+3}$. Similarly, as in the first case, it is not difficult to get that for a sufficiently large $n$
$$
 \frac{4|\cot{\beta}|}{n+3}<0.64
$$
And also taking into account (\ref{psi_s}) and (\ref{psi_c}) of the section \ref{Vspomogatelnie rezultati}, we get that
$$
\left|\sin{\left(\frac{\pi}{3}+\psi_c-\psi_s\right)}\right|<\sin{\left(\frac{\pi}{4}\right)}=\frac{1}{\sqrt{2}}
$$
Then
$$
\begin{aligned}
\left|\frac{4I_1}{n+3}\right|&=\frac{4}{n+3} \dfrac{|\cot{(\beta)}||(\cos{(\frac{\pi}{3}+\psi_c-\psi_s)}\sin{((n+3)c)}-\sin{(\frac{\pi}{3}+\psi_c-\psi_s)}\sinh{((n+3)b)})|}{(\cos{((n+3)c)}+\cosh{((n+3)b)})}\\
&\leq \frac{4}{n+3}\dfrac{|\cot{(\beta)}|(|\cos{(\frac{\pi}{3}+\psi_c-\psi_s)}\sin{((n+3)c)}|+|\sin{(\frac{\pi}{3}+\psi_c-\psi_s)}\cosh{((n+3)b)}|)}{(\cos{((n+3)c)}+\cosh{((n+3)b)})}\\
&< \dfrac{0.64(1+\frac{115}{\sqrt{2}})}{(-1+115)}<0.47
\end{aligned}
$$
So
\begin{equation}
\left|\frac{4I_1}{n+3}\right|<0.47
\label{I_1 vo vtorom sluch}
\end{equation}
Let's estimate $\frac{4I_4}{n+3}$. From \ref{B_s/sin} of the section \ref{Vspomogatelnie rezultati} it follows that:
$$
\frac{B_s}{\sin{(\varphi)}}\leq\frac{B_s(\frac{\pi}{2})}{\sin{(\frac{\pi}{2})}}<1.63
$$
And also taking into account \ref{psi_s} of the section \ref{Vspomogatelnie rezultati}
$$
\left|\sin{\left(\frac{\pi}{3}-\psi_s\right)}\right|\leq\sin{\left(\frac{\pi}{12}\right)}<0.26
$$
It is also obvious that
$$
\frac{4}{(n+3)}{\sin{(\varphi)}}\leq \frac{4}{(n+3)}\frac{(n+3)}{2\pi}<0.64
$$
then
$$
\begin{aligned}
\left|\frac{4I_4}{n+3}\right|&=\frac{4}{n+3}\frac{\cos(\varphi)|\sin{(\beta)}||(\cos{(\psi_s-\frac{\pi}{3})}\sin{((n+3)c)}-\sin{(\psi_s-\frac{\pi}{3})}\sinh{((n+3)b)})|}{\sin^2{(\varphi)}(\cos{((n+3)c)}+\cosh{((n+3)b)})}\\
&\leq \frac{4}{n+3}\frac{|\sin{(\beta)}|(|\cos{(\psi_s-\frac{\pi}{3})}\sin{((n+3)c)}|+|\sin{(\psi_s-\frac{\pi}{3})}\cosh{((n+3)b)}|)}{\sin^2{(\varphi)}(\cos{((n+3)c)}+\cosh{((n+3)b)})}\\
&<1.63*0.64\frac{1+0.26*115}{-1+115}<0.29
\end{aligned}
$$
So
\begin{equation}
\left|\frac{4I_4}{n+3}\right|<0.29
\label{I_4 vo vtorom sluch}
\end{equation}
Then, taking into account (\ref{I_2+I_3 vo vtorom sluch}), (\ref {I_1 vo vtorom sluch}), and (\ref {I_4 vo vtorom sluch}), we get that for $ \varphi \in (\frac {2 \pi}{n + 3}, \frac {\pi}{2})$
$$
\left|\dfrac{2f'}{n+3}\right|=\frac{4}{n+3}|I_1+I_2+I_3+I_4|\leq \frac{4}{n+3}(|I_1|+|I_2+I_3|+|I_4|)<0.8
$$
Which means
$$\left|\frac{4}{n+3}\frac{f'}{f^2+1}\right|<0.8$$
{\it Case 3.} $\varphi \in (\frac{\pi}{2}, \pi)$. In this case, the estimates (\ref{I_2+I_3 vo vtorom sluch}) and (\ref{I_1 vo vtorom sluch}) remain true, which means that the inequalities are met:
\begin{equation}
\left|\frac{4}{(n+3)}\frac{I_1}{1+f^2}\right|\leq\left|\frac{4I_1}{(n+3)}\right|<0.47
\label{I_1 v trtem sluch}
\end{equation}
\begin{equation}
\left|\frac{4}{(n+3)}\frac{I_2+I_3}{1+f^2}\right|\leq \left|\frac{4(I_2+I_3)}{(n+3)}\right|<0.04
\label{I_2+I_3 v trtem sluch}
\end{equation}
Let's estimate $\frac{4I_4}{(n+3)f^2}$.
$$
\begin{aligned}
\left|\frac{4I_4}{(n+3)f^2}\right|&=\left|\dfrac{4 \cos{(\varphi)}(\cos{((n+3)c)}+\cosh{((n+3)b)})}{(n+3)(B \sin{((n+3)c)}+C \sinh{((n+3)b)})}\right|
\end{aligned}
$$
Since $B=B_s\cos{(\frac{\pi}{3}-\psi_s)}$, $C=-B_s\sin{(\frac{\pi}{3}-\psi_s)}$ and $B_s$ is increases, then
$$B<B_s(\pi)<4,$$ and also $$C>B_s\left(\frac{\pi}{2}\right)\sin{\left(\frac{\pi}{3}-\psi_s\left(\frac{\pi}{2}\right)\right)}>0.15$$
From \ref{b} and \ref{b/phi} of the section \ref{Vspomogatelnie rezultati} it follows that
$$
b>b\left(\frac{\pi}{2}\right)>\frac{\pi}{4}
$$
Then
\begin{equation}
\left|\frac{4I_4}{(n+3)f^2}\right|<\frac{4}{(n+3)}\left|\frac{1+\cosh{(\frac{(n+3)\pi}{4})}}{(-4+0.15\sinh{(\frac{(n+3)\pi}{4})})}\right|
\label{vvv}
\end{equation}
Obviously, as $n$ increases, the right-hand side of the inequality (\ref{vvv}) decreases. So when $n>100$
$$\left|\frac{4I_4}{(n+3)f^2}\right|<\frac{4}{(103)}\left|\frac{1+\cosh{(\frac{103\pi}{4})}}{(-4+0.15\sinh{(\frac{103\pi}{4})})}\right|<0.26$$
So
\begin{equation}
\left|\frac{4I_4}{(n+3)(f^2+1)}\right|<\left|\frac{4I_4}{(n+3)f^2}\right|<0.26
\label{I_4 v trtem sluch}
\end{equation}
Then, taking into account (\ref{I_1 v trtem sluch}), (\ref{I_2+I_3 v trtem sluch}) and (\ref{I_4 v trtem sluch}), we get
$$
\begin{aligned}
\left|\dfrac{2f'}{(n+3)(1+f^2)}\right|&=\frac{4}{n+3}\dfrac{|I_1+I_2+I_3+I_4|}{(1+f^2)}\leq \frac{4|I_1|}{(n+3)(1+f^2)}+\frac{4|I_2+I_3|}{(n+3)(1+f^2)}+\frac{4|I_4|}{(n+3)(1+f^2)}\\
&<0.04+0.47+0.26=0.77<0.8
\end{aligned}
$$
The inequality (\ref{2333}) is proved in the same way as in the cases 2 and 3 above.
\end{proof}
Let's introduce two new functions:
$$F(\varphi,n):=\frac{n+3}{2}\varphi- \arctan{(f(\varphi))}$$
$$G(\varphi,n):=\frac{n+3}{2}\varphi- \frac{\pi}{2}+\arctan{(h(\varphi))}$$
\begin{proof}[Proof of the theorem \ref{theorem 2}]
	$$
	F'(\varphi, n)=\frac{n+3}{2} - \frac{f'(\varphi)}{1+f^2(\varphi)}.
	$$
	From the lemma \ref{lemma ob ogran proizv} it follows that the recurrent formula (\ref{Rekurent formula 1}) converges. If $j=1,2,\dots , \left[ \frac{n+1}{2}\right]$ then
	$F(\frac{\pi}{n+3},n)=\frac{\pi}{2}-\arctan{(f(\frac{\pi}{n+3}))}<\pi j$ , $F(\pi, n)=\frac{(n+3)\pi}{2}-\arctan{(f(\pi))}>\pi j$. Which means that the equation
	$F(\varphi,n)=\pi j$, has exactly one root on the interval $(\frac{\pi}{n+3}, \pi)$, for each $j=1,2,\dots , \left[ \frac{n+1}{2}\right]$. From which follows the statement of the first part of the theorem.\\
	$$
	G'(\varphi, n)=\frac{n+3}{2} - \frac{h'(\varphi)}{1+h^2(\varphi)}.
	$$
	From the lemma \ref{lemma ob ogran proizv} it follows that the recurrent formula (\ref{Rekurent formula 2}) converges. And also that $G'(\varphi, n)>0$ for $\varphi \in (\frac{2\pi}{n+3}, \pi)$. Then if  $j=1,2,\dots , \left[ \frac{n}{2}\right]$ then $G(\frac{2\pi}{n+3},n)=\frac{\pi}{2}+\arctan{(h(\frac{\pi}{n+3}))}<\pi j$, $G(\pi,n)=\frac{(n+2)\pi}{2}+\arctan{(h(\pi))}>\pi j$. It means that the equation
	$G(\varphi,n)=\pi j$, has exactly one root on the interval  $(\frac{\pi}{n+3}, \pi)$, for each $j=1,2,\dots , \left[ \frac{n}{2}\right]$. From which follows the statement of the second part of the theorem.
\end{proof}
\begin{proof}[Proof of the lemma \ref{lemma2}]
We show that all the roots of the equation are different. To do this it is sufficient to show that the following inequalities hold for all roots $\varphi_{2j-1}<\varphi_{2j}$ and $\varphi_{2j}<\varphi_{2j+1}$. We prove the first inequality. Suppose this is incorrect, then
$$
\frac{2}{n+3}\left[\pi j +\arctan{(f(\varphi_{2j-1}))}\right]\geq \frac{2}{n+3}\left[\pi j+\frac{\pi}{2} -\arctan{(h(\varphi_{2j}))}\right]
$$
which is the same as
\begin{equation}
	\arctan{(f(\varphi_{2j-1}))} \geq \frac{\pi}{2} -\arctan{(h(\varphi_{2j}))}
	\label{nnn}
\end{equation}
Since $\arctan x\in [-\frac{\pi}{2},\frac{\pi}{2}]$ then to perform the inequality [\ref{nnn}], it is necessary that
$$
h(\varphi_{2j})\geq 0
$$
which is equivalent to
$$
B(\varphi_{2j})\sin{((n+3)a(\varphi_{2j}))}\geq C
(\varphi_{2j})\sinh{((n+3)b(\varphi_{2j}))}
$$
From \ref{b/phi} it follows, that $\frac{b(\varphi)}{\varphi}\geq \frac{b(\pi)}{\pi}>\frac{1}{2}$, which means $b(\varphi_{2j})> \frac{\varphi_{2j}}{2}$, then we get the inequality
$$
B(\varphi_{2j})>C(\varphi_{2j})\sinh{\left(\frac{n+3}{2} \varphi_{2j}\right)}
$$
which is equivalent to
$$
\dfrac{B(\varphi_{2j})}{C(\varphi_{2j})}>\sinh{\left(\frac{n+3}{2} \varphi_{2j}\right)}
$$
Since
$$
\dfrac{B(\varphi_{2j})}{C(\varphi_{2j})}=-\dfrac{\Re{(\sin{(\beta)}e^{\frac{-\pi i}{3}})}}{\Im{(\sin{(\beta)}e^{\frac{-\pi i}{3}})}}=-\dfrac{B_s\Re{(e^{\frac{-\pi i}{3}+i\psi_s})}}{B_s\Im{(e^{\frac{-\pi i}{3}+i\psi_s})}}=\cot{\left(\frac{\pi}{3}-\psi_s\right)}
$$
we get the inequality
\begin{equation}
	\cot{\left(\frac{\pi}{3}-\psi_s\right)}>\sinh{\left(\frac{n+3}{2} \varphi_{2j}\right)}
	\label{eee}
\end{equation}
From \ref{psi_s} it follows that the left side of the equation (\ref{eee}) decreases, and the right side increases. Also note that the left part does not depend on $n$, and the right part grows with increasing $n$, respectively, with a sufficiently large $n$ $\varphi_{2j}$ and therefore $\varphi_{2j-1}$ (since $\varphi_{2j-1}-\varphi_{2j}<\dfrac{\pi}{n+3}$) can be made arbitrarily small. Let
\begin{equation}
	\varphi_{2j}<\varphi_{2j-1}<\dfrac{\pi}{12}.
	\label{mnb}
\end{equation}
In fact, this will be true already at $n>7$.
Taking into account (\ref{mnb}), find the upper estimate for $|f(\varphi_{2j-1})|$ and $|h(\varphi_{2j})|$.
$$
\dfrac{B(\varphi)}{\sin{(\varphi)}}=\dfrac{B_s\cos{\left(\frac{\pi}{3}-\psi_s\right)}}{\sin{(\varphi)}}
$$
Then according to \ref{B_s/sin} we have
\begin{equation}
	\dfrac{B(\varphi_{2j})}{\sin{(\varphi_{2j})}}= 	\dfrac{B_s(\varphi_{2j})\cos{\left(\frac{\pi}{3}-\psi_s\right)}}{\sin{(\varphi_{2j})}}\leq \dfrac{B_s(\varphi_{2j})}{\sin{(\varphi_{2j})}}<\dfrac{B_s(\frac{\pi}{12})}{\sin{(\frac{\pi}{12})}}<1.1
	\label{ocB}
\end{equation}
$$
\dfrac{C(\varphi)}{\sin{(\varphi)}}=\dfrac{B_s\sin{\left(\frac{\pi}{3}-\psi_s\right)}}{\sin{(\varphi)}}
$$
Then according to \ref{B_s/sin} and \ref{psi_s} we have
\begin{equation}
	\dfrac{C(\varphi_{2j})}{\sin{(\varphi_{2j})}}=	\dfrac{C_s(\varphi_{2j})\sin{\left(\frac{\pi}{3}-\psi_s\right)}}{\sin{(\varphi_{2j})}}<\dfrac{B_s(\frac{\pi}{12})\sin{\left(\frac{\pi}{3}-\psi_s(\frac{\pi}{12})\right)}}{\sin{(\frac{\pi}{12})}}<0.01
	\label{ocC}
\end{equation}
Then
$$
\begin{aligned}
	|h(\varphi_{2j})|<& 2\left|\frac{B(\varphi_{2j})\sin{((n+3)c(\varphi_{2j}))}-C(\varphi_{2j})\sinh{((n+3)b(\varphi_{2j}))}}{\sin{(\varphi_{2j})}(-\cos{((n+3)c(\varphi_{2j}))}+\cosh{((n+3)b(\varphi_{2j}))})}\right|\\
	<&\dfrac{2.2}{-1+\cosh{((n+3)b(\varphi_{2j}))}}+\dfrac{0.02\cosh{((n+3)b(\varphi_{2j}))}}{-1+\cosh{((n+3)b(\varphi_{2j}))}}
\end{aligned}
$$
Using the estimate (\ref{ocenka cosh v pervom sluchae}), we get that
\begin{equation}
	|h(\varphi_{2j})|<\dfrac{2.2}{6.6}+\dfrac{0.02*7.6}{6.6}<0.5<\dfrac{1}{\sqrt{3}}
\end{equation}
From similar reasoning, we get that
\begin{equation}
	|f(\varphi_{2j-1})|<\dfrac{2.2}{6.6}+\dfrac{0.02*7.6}{6.6}<0.5<\dfrac{1}{\sqrt{3}}
\end{equation}
Which contradicts the inequality (\ref{nnn}), which means $ \varphi_{2 j-1}<\varphi_{2 j}$. Similarly, it is shown that $\varphi_{2j}<\varphi_{2j+1}$.
\end{proof}
\begin{proof}[Proof of the theorem \ref{Teorema sluch 1}]
	Denote by $j_m$ the smallest $j$ for which the inequality $\frac{1}{2}e^{\pi(j-1)}>q^2$ is satisfied. Let  $d_{1,j}=\frac{\pi j}{q}$. Since $2b> \varphi$, as shown in the \ref{b/phi} of section \ref{Vspomogatelnie rezultati}, then $$\cosh{(2qb)}>\frac{1}{2}e^{d_{1,j}-\frac{u}{q}}>\frac{1}{2}e^{\pi(j_m-1)}>q^2.$$
	In this case
	 $$\frac{\sin{(2qc)}}{\cosh{(2qb)}}=O\left(\frac{1}{q^2}\right),$$
	 $$\frac{\cos{(2qc)}}{\cosh{(2qb)}}=O\left(\frac{1}{q^2} \right),$$
	  and
	  $$\tanh({(2qb)})=1+O\left(\frac{1}{q^2}\right).$$
	  Then the equation (\ref{Glavnoe ur U}) can be rewritten as:
	\begin{equation}
	u=\arctan\left({2\frac{C(d_{1,j}+\frac{u}{q})+O(\frac{1}{q^2})}{\sin{(d_{1,j}+\frac{u}{q})}(1+O(\frac{1}{q^2}))}}\right)
	\label{ur1}
	\end{equation}
	Let's consider two cases. {\it Case 1.} $d_{1,j}<(1-\varepsilon)\pi$, where $\varepsilon$ small positive number. Then, assuming $u=u_1+\frac{u_2}{q}$, the equation (\ref{ur1}) takes the form:
	$$
	u_1+\frac{u_2}{q}=\arctan\left({2\frac{C(d_{1,j}+\frac{u_1}{q}+O(\frac{1}{q^2}))+O(\frac{1}{q^2})}{\sin{(d_{1,j}+\frac{u_1}{q}+O(\frac{1}{q^2}))}+O(\frac{1}{q^2})}}\right)
	$$
\begin{multline*}
	\arctan\left(2\frac{C(d_{1,j}+\frac{u_1}{q}+O(\frac{1}{q^2}))+O(\frac{1}{q^2})}{\sin{(d_{1,j}+\frac{u_1}{q}+O(\frac{1}{q^2}))}+O(\frac{1}{q^2})}\right)\\=
   \arctan\left(2\frac{C(d_{1,j})}{\sin{(d_{1,j})}}+2\frac{C'(d_{1,j})\sin{(d_{1,j})}-C(d_{1,j})\cos{(d_{1,j})}}{\sin^2{(d_{1,j})}}\frac{u_1}{q}+O\left(\frac{1}{q^2}\right)\right)	\\
	= \arctan{\left(2\frac{C(d_{1,j})}{\sin{(d_{1,j})}} \right)}+2\frac{\sin^2{(d_{1,j})}}{\sin^2{(d_{1,j})}+4C^2(d_{1,j})}\frac{C'(d_{1,j})\sin{(d_{1,j})}-C(d_{1,j})\cos{(d_{1,j})}}{\sin^2{(d_{1,j})}}\frac{u_1}{q}+O\left(\frac{1}{q^2}\right)
\end{multline*}
	In this case, assuming $$u_1^{\star}=\arctan{\left(2\frac{C(d_{1,j})}{\sin{(d_{1,j})}} \right)},$$ and
	\\$$u_2^{\star}=2\frac{C'(d_{1,j})\sin{(d_{1,j})}-C(d_{1,j})\cos{(d_{1,j})}}{\sin^2{(d_{1,j})}+4C^2(d_{1,j})}\arctan{\left(2\frac{C(d_{1,j})}{\sin{(d_{1,j})}}\right)},$$
	we get that $|u_1-u_1^{\star}|=O(\frac{1}{q^2})$ and $|u_2-u_2^{\star}|=O(\frac{1}{q})$\\
	{\it Case 2.}$(1-\varepsilon)\pi \leq d_{1,j}<\pi$. Since $C(d_{1,j}+\frac{u}{q})>0$ and $\sin{(d_{1,j}+\frac{u}{q})}>0$, then the equation (\ref{ur1}) can be rewritten as:
	$$
	u_1+\frac{u_2}{q}=\arccot{\left(\frac{1}{2}\frac{\sin{(d_{1,j}+\frac{u_1}{q}+O(\frac{1}{q^2}))}+O(\frac{1}{q^2})}{C(d_{1,j}+\frac{u_1}{q}+O(\frac{1}{q^2}))+O(\frac{1}{q^2})}\right)}
	$$
	then similarly to the first case:\\
	\begin{multline*}
	\arccot{\left(\frac{1}{2}\frac{\sin{(d_{1,j}+\frac{u_1}{q}+O(\frac{1}{q^2}))}+O(\frac{1}{q^2})}{C(d_{1,j}+\frac{u_1}{q}+O(\frac{1}{q^2}))+O(\frac{1}{q^2})}\right)}\\
	=
	\arccot{\left(\frac{1}{2}\frac{\sin{(d_{1,j})}}{C(d_{1,j})}+\frac{1}{2}\frac{C'(d_{1,j})\cos{(d_{1,j})}-C(d_{1,j})\sin{(d_{1,j})}}{C^2(d_{1,j})}\frac{u_1}{q}+O\left(\frac{1}{q^2}\right)\right)}\\
	=
	\arccot{\left(\frac{1}{2}\frac{\sin{(d_{1,j})}}{C(d_{1,j})}\right)}-\frac{1}{2}\frac{C^2(d_{1,j})}{4C^2(d_{1,j})+\sin^2{(d_{1,j})}}\frac{C'(d_{1,j})\cos{(d_{1,j})}-C(d_{1,j})\sin{(d_{1,j})}}{C^2(d_{1,j})}\frac{u_1}{q}+O\left(\frac{1}{q^2}\right).
	\end{multline*}
	As a result we get the same result as in the first case.\\
	The second part of the theorem is proved in a similar way.
\end{proof}
\begin{lm}
	\label{lemma o shod 2}
	\begin{enumerate}
		\item  If $d_{1,j}: \:\frac{1}{2}e^{\pi(j-1)}\leq q^2$ then, for a sufficiently large $n$, the equation
		$$u_1=\arctan{\left(Z_1^{(1)}\right)}$$
		has a unique solution on the interval $(-\frac{\pi}{2},\frac{\pi}{2})$.
		\item If $d_{2,j}: \:\frac{1}{2}e^{\pi(j-1)}\leq q^2$ then, for a sufficiently large $n$, the equation
		$$w_1=-\arctan{\left(Z_1^{(2)}\right)}$$
		has a unique solution on the interval $(-\frac{\pi}{2},\frac{\pi}{2})$.
	\end{enumerate}
	\begin{proof}
		Considering that $(P_a^{(1)})'=1$, $(P_b^{(1)})'=\sqrt{3}$
		$$
		\begin{aligned}
		(Z^{(1)})'=2\dfrac{((1+\frac{3}{16}d_{1,j}^2)\cos{(P_a^{(1)})}+\frac{3}{16}d_{1,j}^2\cosh{(P_b^{(1)})})(\cos{(P_a^{(1)})}+\cosh{(P_b^{(1)})})}{(\cos{(P_a^{(1)})}+\cosh{(P_b^{(1)})})^2}\\
		+2\frac{((1+\frac{3}{16}d_{1,j}^2)\sin{(P_a^{(1)})}+\frac{\sqrt{3}}{16}d_{1,j}^2\sinh{(P_b^{(1)})})(-\sin{(P_a^{(1)})}+\sqrt{3}\sinh{(P_b^{(1)})}}{(\cos{(P_a^{(1)})}+\cosh{(P_b^{(1)})})^2}
		\end{aligned}
		$$
	Expending the brackets and giving similar terms we get:
		$$
		(Z^{(1)})'=2\dfrac{(1+\frac{3}{8}d_{1,j}^2)+(1+\frac{3}{8}d_{1,j}^2)\cos{(P_a^{(1)})}\cosh{(P_b^{(1)})}-(1+\frac{\sqrt{3}}{8}d_{1,j}^2)\sinh{(P_b^{(1)})}\sin{(P_a^{(1)})} }{(\cos{(P_a^{(1)})}+\cosh{(P_b^{(1)})})^2}\\
		$$
From the condition of the lemma, it follows that for a sufficiently large $n$ $d_{1,j}=\frac{\pi j}{q}<0.1$,  then $\frac{\sqrt{3}}{8}d_{1,j}^2<\frac{3}{8}d_{1,j}^2<0.01$, this means that the following inequality is true:
	\begin{equation}
		(Z^{(1)})'<2\dfrac{1.01+1.01\cosh{(P_b^{(1)})}+1.01\sinh{(P_b^{(1)})} }{(\cos{(P_a^{(1)})}+\cosh{(P_b^{(1)})})^2}<\dfrac{2.02(1+2\cosh{(P_b^{(1)})}) }{(-1+\cosh{(P_b^{(1)})})^2}
		\label{000}
	\end{equation}
	Let's estimate $\cosh{(P_b^{(1)})}$.
	$$
	P_b^{(1)}=\sqrt{3}(\pi j - u_1)\leq \sqrt{3} (\pi -\frac{\pi}{2}) =\frac{ \sqrt{3}\pi}{2}
	$$
	Then
	$$
	\cosh{(P_b^{(1)})}<\cosh(\frac{ \sqrt{3}\pi}{2})<7.6
	$$
	Obviously, the right-hand side of the inequality (\ref{000}) decreases with the growth of $ \cosh{(P_b^{(1)})}$, then
$$
	(Z^{(1)})'<\dfrac{2.02(1+15.2)}{6.6^2}<0.76
$$
		From this, first, it follows that the recurrent formula $u_1^{(k+1)}=\arctan{(Z_1^{(1)}(u_1^{(k+1})))}$ converges. Second, that the function $H(u_1)=u_1-\arctan{(Z_1^{(1)})}$ increases. Since $H(-\frac{\pi}{2})<0$, $H(\frac{\pi}{2})>0$ then  the root exists, and given that $H(u_1)$ is increasing function then this root is the only one.
		The second part of the lemma is proved similarly.
	\end{proof}
\end{lm}
\begin{proof}[Proof ot the theorem \ref{theorem sluch 2}]
Let $d_{1,j}=\frac{\pi j}{q}$, $q=\frac{n+3}{2}$, $u=u_1+\frac{u_2}{q}$.
We have the equation:\\
\begin{equation}
u=\arctan{f}
\label{Ur1}
\end{equation}
where
$$
f=2\frac{B(\varphi)\sin{(2qc)}+C(\varphi)\sinh{(2qb)}}{\sin(\varphi)(\cos{(2qc)}+\cosh{(2qb)})}
$$
and the following inequality is also true
\begin{equation}
\frac{1}{2}e^{\pi(j-1)}\leq q^2
\label{neravenstvo}
\end{equation}
From \ref{B v 0} of the section \ref{Vspomogatelnie rezultati} it is not difficult to get that
\begin{equation}
\frac{B(\varphi)}{\sin(\varphi)}=1+\frac{3}{16}\varphi^2+O(\varphi^4)
\end{equation}
From the inequality (\ref{neravenstvo}), it follows that $O(\varphi^4)=O(\frac{1}{q^2})$.
$\varphi=d_{1,j}+\frac{u_1}{q}+O(\frac{1}{q^2})$, so
\begin{equation}
\frac{B(\varphi)}{\sin(\varphi)}=1+\frac{3d_{1,j}^2}{16}+\frac{3d_{1,j} u_1}{8}\frac{1}{q}+O\left(\frac{1}{q^2}\right)=B_1^{(1)}+B_2^{(1)}\frac{1}{q}+O\left(\frac{1}{q^2}\right),
\end{equation}
where $$B_1^{(1)}=1+\frac{3d_{1,j}^2}{16}, \;\; B_2^{(1)}=\frac{3d_{1,j} u_1}{8}$$
Similarly
\begin{equation}
\frac{C(\varphi)}{\sin(\varphi)}=\frac{\sqrt{3}d_{1,j}^2}{16}+\frac{\sqrt{3}d_{1,j} u_1}{8}\frac{1}{q}+O\left(\frac{1}{q^2}\right)=C_1^{(1)}+C_2^{(1)}\frac{1}{q}+O\left(\frac{1}{q^2}\right),
\end{equation}
where $$C_1^{(1)}=\frac{\sqrt{3}d_{1,j}^2}{16}, \;\; C_2^{(1)}=\frac{\sqrt{3}d_{1,j} u_1}{8}$$
In \ref{a v 0} of the section \ref{Vspomogatelnie rezultati}, it was shown that
$$c(\varphi)=\frac{1}{2}\varphi-\frac{1}{16}\varphi^3+O(\varphi^5)$$
\begin{equation}
\begin{aligned}
\varphi^3=\left(d_{1,j}+\frac{u_1}{q}+\frac{u_2}{q^2}+O\left(\frac{1}{q^3}\right)\right)^3= d_{1,j}^3+3d_{1,j}^2\frac{u_1}{q}+3d_{1,j}\frac{u_1^2}{q^2}+3d_{1,j}^2\frac{u_2}{q^2}+3d_{1,j}\frac{u_2^2}{q^4}+\\
6d_{1,j}\frac{u_1}{q}\frac{u_2}{q^2}+\frac{u_1^3}{q^3}+3\frac{u_1^2}{q^2}\frac{u_2}{q^2}+3\frac{u_1}{q}\frac{u_2^2}{q^4}+\frac{u_2^3}{q^6}+O\left(\frac{1}{q^3}\right)
\end{aligned}
\label{01}
\end{equation}
From the inequality (\ref{neravenstvo}), it follows that for a sufficiently large $n$, on the right side of the equality (\ref{01}), all the terms after the third one are $O(\frac{1}{q^3})$. Then
\begin{equation}
\begin{aligned}
2qa=d_{1,j}q+u_1+(u_2-\frac{1}{8}(d_{1,j}^3q^2+3d_{1,j}^2u_1q+3d_{1,j}u_1^2))\frac{1}{q}+ O\left(\frac{1}{q^2}\right)\\ = P_a^{(1)}+(u_2+Q_a^{(1)})\frac{1}{q}+ O\left(\frac{1}{q^2}\right),
\end{aligned}
\end{equation}
where $$P_a^{(1)}=d_{1,j}q+u_1, \;\; \;\; Q_a^{(1)}=-\frac{1}{8}(d_{1,j}^3q^2+3d_{1,j}^2u_1q+3d_{1,j}u_1^2).$$
In \ref{b v 0} of the section \ref{Vspomogatelnie rezultati}, it was shown that
$$
b(\varphi)=\frac{\sqrt{3}}{2}\varphi- \frac{\sqrt{3}}{48}\varphi^3+O(\varphi^5)
$$
Then from similar reasoning:
\begin{equation*}
2qb=P_b^{(1)}+(\sqrt{3}u_2+Q_b^{(1)})\frac{1}{q}+ O\left(\frac{1}{q^2}\right),
\end{equation*}
where
 $$P_b^{(1)}=\sqrt{3}(d_{1,j}q+u_1),\;\;\;\; Q_b^{(1)}=-\frac{\sqrt{3}}{24}(d_{1,j}^3q^2+3d_{1,j}^2u_1q+3d_{1,j}u_1^2).$$
The numerator and denominator of the function $f$ are divided by $ \cos{P_b}$ and expend each of the terms in a Taylor series to  $O\left(\frac{1}{q^2}\right)$
$$
\frac{\sin(P_a^{(1)}+(Q_a^{(1)}+u_2)\frac{1}{q}+O(\frac{1}{q^2}))}{\cosh{P_b^{(1)}}}=\frac{\sin{P_a^{(1)}}}{\cosh{P_b^{(1)}}}+\frac{\cos{P_a^{(1)}}}{\cosh{P_b^{(1)}}}\frac{Q_a^{(1)}+u_2}{q}+O\left(\frac{1}{q^2}\right)
$$
$$
\begin{aligned}
&\frac{\sin(P_a^{(1)}+(Q_a^{(1)}+u_2)\frac{1}{q}+O(\frac{1}{q^2}))}{\cosh{P_b^{(1)}}}\frac{B(\varphi)}{\sin(\varphi)}
\\&=\left[\frac{\sin{P_a^{(1)}}}{\cosh{P_b^{(1)}}}+\frac{\cos{P_a^{(1)}}}{\cosh{P_b^{(1)}}}\frac{Q_a^{(1)}+u_2}{q}+O\left(\frac{1}{q^2}\right)\right]\left[B_1^{(1)}+B_2^{(1)}\frac{1}{q}+O\left(\frac{1}{q^2}\right) \right]\\
&=\frac{\sin{P_a^{(1)}}}{\cosh{P_b^{(1)}}}B_1^{(1)}+\left(\frac{\sin{P_a^{(1)}}}{\cosh{P_b^{(1)}}}B_2^{(1)}+B_1^{(1)}\frac{Q_a^{(1)}\cos{P_a^{(1)}}}{\cosh{P_b^{(1)}}} + B_1^{(1)}\frac{u_2\cos{P_a^{(1)}}}{\cosh{P_b^{(1)}}}\right)\frac{1}{q}+O\left(\frac{1}{q^2}\right)
\end{aligned}
$$
\\
$$
\frac{\sinh(P_b^{(1)}+(Q_b^{(1)}+\sqrt{3}u_2)\frac{1}{q}+o(\frac{1}{q}))}{\cosh{P_b^{(1)}}}=\frac{\sinh{P_b^{(1)}}}{\cosh{P_b^{(1)}}}+\frac{Q_b^{(1)}+\sqrt{3}u_2}{q}+O\left(\frac{1}{q^2}\right)
$$
$$
\begin{aligned}
&\frac{\sinh(P_b^{(1)}+(Q_b^{(1)}+\sqrt{3}u_2)\frac{1}{q}+O(\frac{1}{q^2}))}{\cosh{P_b^{(1)}}}\frac{C(\varphi)}{\sin(\varphi)}\\
& = \left[\frac{\sinh{P_b^{(1)}}}{\cosh{P_b^{(1)}}}+\frac{Q_b^{(1)}+\sqrt{3}u_2}{q}+O\left(\frac{1}{q^2}\right)\right]\left[C_1^{(1)}+C_2^{(1)}\frac{1}{q}+O\left(\frac{1}{q^2}\right) \right]\\
&=\frac{\sinh{P_b^{(1)}}}{\cosh{P_b^{(1)}}}C_1^{(1)}+\left(\frac{\sinh{P_b^{(1)}}}{\cosh{P_b^{(1)}}}C_2^{(1)}+C_1^{(1)}Q_b^{(1)} + C_1^{(1)}\sqrt{3}u_2\right)\frac{1}{q}+O\left(\frac{1}{q^2}\right)
\end{aligned}
$$
\\
Then the numerator of the function $f$ will take the form:
$$
X_1^{(1)}+(X_2^{(1)}+X_3^{(1)}u_2)\frac{1}{q}+O\left(\frac{1}{q^2}\right)
$$
where
$$
X_1^{(1)}=2\left(\frac{\sin{P_a^{(1)}}}{\cosh{P_b^{(1)}}}B_1^{(1)}+\frac{\sinh{P_b^{(1)}}}{\cosh{P_b^{(1)}}}C_1^{(1)}\right)
$$
$$
X_2^{(1)}=2\left(\frac{\sin{P_a^{(1)}}}{\cosh{P_b^{(1)}}}B_2^{(1)}+B_1^{(1)}\frac{Q_a^{(1)}\cos{P_a^{(1)}}}{\cosh{P_b^{(1)}}}+\frac{\sinh{P_b^{(1)}}}{\cosh{P_b^{(1)}}}C_2^{(1)}+C_1^{(1)}Q_b^{(1)}\right)
$$
$$
X_3^{(1)}=2\left(B_1^{(1)}\frac{\cos{P_a^{(1)}}}{\cosh{P_b^{(1)}}}+C_1^{(1)}\sqrt{3}\right)
$$
Consider the denominator of the function $f$.
$$
\begin{aligned}
1+\frac{\cos(P_a^{(1)}+(Q_a^{(1)}+u_2)\frac{1}{q}+O(\frac{1}{q^2}))}{\cosh{P_b^{(1)}}}&=1+\frac{\cos{P_a^{(1)}}}{\cosh{P_b^{(1)}}}-\frac{\sin{P_a^{(1)}}}{\cosh{P_b^{(1)}}}\frac{Q_a^{(1)}+u_2}{q}+O\left(\frac{1}{q^2}\right)\\
&=Y_1^{(1)}+(Y_2^{(1)}+Y_3^{(1)}u_2)\frac{1}{q}+O\left(\frac{1}{q^2}\right),
\end{aligned}
$$
where
$$
Y_1^{(1)}=1+\frac{\cos{P_a^{(1)}}}{\cosh{P_b^{(1)}}}
$$
$$
Y_2^{(1)}=-\frac{Q_a^{(1)}\sin{P_a^{(1)}}}{\cosh{P_b^{(1)}}}+\frac{Q_b^{(1)}\sinh{P_b^{(1)}}}{\cosh{P_b^{(1)}}}
$$
$$
Y_3^{(1)}=-\frac{\sin{P_a^{(1)}}}{\cosh{P_b^{(1)}}}+\frac{\sqrt{3}\sinh{P_b^{(1)}}}{\cosh{P_b^{(1)}}}
$$
Then the function $f$ will take the form
$$
f=\frac{X_1^{(1)}+(X_2^{(1)}+X_3^{(1)}u_2)\frac{1}{q}+O(\frac{1}{q^2})}{Y_1^{(1)}+(Y_2^{(1)}+Y_3^{(1)}u_2)\frac{1}{q}+O(\frac{1}{q^2})}=Z_1^{(1)}+(Z_2^{(1)}+Z_3^{(1)}u_2)\frac{1}{q}+O\left(\frac{1}{q^2}\right)
$$
where
$$
Z_1^{(1)}=\frac{X_1^{(1)}}{Y_1^{(1)}}
$$
$$
Z_2^{(1)}=\frac{X_2^{(1)}Y_1^{(1)}-X_1^{(1)}Y_2^{(1)}}{(Y_1^{(1)})^2}
$$
$$
Z_3^{(1)}=\frac{X_3^{(1)}Y_1^{(1)}-X_1^{(1)}Y_3^{(1)}}{(Y_1^{(1)})^2}
$$
As a result, we have the equation:
$$u_1+\frac{u_2}{q}=\arctan{\left(Z_1^{(1)}+(Z_2^{(1)}+Z_3^{(1)}u_2)\frac{1}{q}+O\left(\frac{1}{q^2}\right)\right)}$$
expand $\arctan$ in a Taylor series:
$$u_1+\frac{u_2}{q}=\arctan{Z_1^{(1)}}+\frac{Z_2^{(1)}+Z_3^{(1)}u_2}{1+(Z_1^{(1)})^2}\frac{1}{q}+O\left(\frac{1}{q^2}\right)$$
then $u_1$, is defined from the equation
$$
u_1=\arctan{Z_1^{(1)}}
$$
$$
u_2=\frac{Z_2^{(1)}}{1+(Z_1^{(1)})^2-Z_3^{(1)}}=R^{(1)}(w_1)
$$
Consider the second equation.
Let $d_{2,j}=\frac{\pi j+\frac{\pi}{2}}{q}$, $q=\frac{n+3}{2}$, $w=w_1+\frac{w_2}{q}$.
We have the equation:\\
\begin{equation}
w=\arctan{(-h)}
\label{Ur2}
\end{equation}
where
$$
-h=2\frac{-B(\varphi)\sin{(2qc)}+C(\varphi)\sinh{(2qb)}}{\sin(\varphi)(-\cos{(2qc)}+\cosh{(2qb)})}
$$
Performing similar operations we get the following\\
 $$P_a^{(2)}=d_{2,j}q+w_1, \;\; \;\; Q_a^{(2)}=-\frac{1}{8}(d_{2,j}^3q^2+3d_{2,j}^2w_1q+3d_{2,j}w_1^2).$$
 $$P_b^{(2)}=\sqrt{3}(d_{2,j}q+w_1),\;\;\;\; Q_b^{(1)}=-\frac{\sqrt{3}}{24}(d_{2,j}^3q^2+3d_{2,j}^2w_1q+3d_{2,j}w_1^2).$$
$$B_1^{(2)}=1+\frac{3d_{2,j}^2}{16}$$ $$B_2^{(2)}=\frac{3d_{2,j} w_1}{8}$$\\
$$C_1^{(2)}=1+\frac{\sqrt{3}d_{2,j}^2}{16}$$, $$C_2^{(2)}=\frac{\sqrt{3}d_{2,j} w_1}{8}$$
$$
X_1^{(2)}=2\left(-\frac{\sin{P_a^{(2)}}}{\cosh{P_b^{(2)}}}B_1^{(2)}+\frac{\sinh{P_b^{(2)}}}{\cosh{P_b^{(2)}}}C_1^{(2)}\right)
$$
$$
X_2^{(2)}=2\left(-\frac{\sin{P_a^{(2)}}}{\cosh{P_b^{(2)}}}B_2^{(2)}-B_1^{(2)}\frac{Q_a^{(2)}\cos{P_a^{(2)}}}{\cosh{P_b^{(2)}}}+\frac{\sinh{P_b^{(2)}}}{\cosh{P_b^{(2)}}}C_2^{(2)}+C_1^{(2)}Q_b^{(2)}\right)
$$
$$
X_3^{(2)}=2\left(-B_1^{(2)}\frac{\cos{P_a^{(2)}}}{\cosh{P_b^{(2)}}}+C_1^{(2)}\sqrt{3}\right)
$$
$$
Y_1=1-\frac{\cos{P_a^{(2)}}}{\cosh{P_b^{(2)}}}
$$
$$
Y_2^{(2)}=\frac{Q_a^{(2)}\sin{P_a^{(2)}}}{\cosh{P_b^{(2)}}}+\frac{Q_b^{(2)}\sinh{P_b^{(2)}}}{\cosh{P_b^{(2)}}}
$$
$$
Y_3^{(2)}=\frac{\sin{P_a^{(2)}}}{\cosh{P_b^{(2)}}}+\frac{\sqrt{3}\sinh{P_b^{(2)}}}{\cosh{P_b^{(2)}}}
$$
then
$$
-h=\arctan{\left(\frac{X_1^{(2)}+(X_2^{(2)}+X_3^{(2)}w_2)\frac{1}{q}+o(\frac{1}{q})}{Y_1^{(2)}+(Y_2^{(2)}+Y_3^{(2)}w_2)\frac{1}{q}+o(\frac{1}{q})}\right)}=Z_1^{(2)}+(Z_2^{(2)}+Z_3^{(2)}u_2)\frac{1}{q}+O\left(\frac{1}{q^2}\right)
$$
and we get similar expressions for the second equation:
$$
w_1=\arctan{Z_1^{(2)}},
$$
$$
w_2=\frac{Z_2^{(2)}}{1+(Z_1^{(2)})^2-Z_3^{(2)}}=R^{(2)}(w_1).
$$

\end{proof}
\begin{proof}[Proof of the theorem \ref{posled teorema}]
	$$
	\begin{aligned}
	\lambda_{2j-1}^{(n)}=g(\varphi_{2j-1}^{(n)})=g\left( d_{1,j}+\frac{2u_1^{\star}}{n+3}+\frac{4u_2^{\star}}{(n+3)^2}+O\left(\frac{1}{n^3}\right)\right)\\=
	g(d_{1,j})+g'(d_{1,j})\left(\frac{2u_1^{\star}}{n+3}+\frac{4u_2^{\star}}{(n+3)^2}+O\left(\frac{1}{n^3}\right)\right)
	\\+\frac{1}{2}g''(d_{1,j})\left(\frac{2u_1^{\star}}{n+3}+\frac{4u_2^{\star}}{(n+3)^2}+O\left(\frac{1}{n^3}\right)\right)^2+O\left(\frac{1}{n^3}\right)
	\end{aligned}
	$$
	Expending the brackets and leaving the terms of order no more than $O\left(\frac{1}{n^2}\right)$, we get the statement of the first part of the theorem. The second part is proved similarly.
\end{proof}
\begin{proof}[Proof of theorem \ref{extreme eigenvalues}]
	We know that $\lambda_j^{(n)}=g(\varphi_j^{(n)})$ for all $j$ and $n$.
	\begin{enumerate}[$i)$]
		\item Given that $q=\frac{n+3}{2}$ and $\varphi_{2j-1}^{(n)}=d_{i,j}+\frac{u_1^*}{q}+\frac{u_2^*}{q^2}+O\left(\frac{1}{q^3}\right)$ we get
		\[
		\lambda_{2j-1}^{(n)}=m\sin^6\left(\frac{d_{i,j}}{2}+\frac{u_1^*}{2q}+\frac{u_2^*}{2q^2}+O\left(\frac{1}{q^3}\right)\right)\ \ (q\to\infty).
		\] 
		
		Let $C_j:=1+\frac{u_1^*}{\pi j}$. Since $\sin x=x-\frac{1}{6}x^3+O(x^5)$, $x\to0$, a simple calculation shows that
		\begin{align*}
			\lambda_{2j-1}^{(n)}&=\frac{mC_{j}^6d_{1,j}^6}{2^6}\left[1+\frac{u_2^*}{d_{1,j}C_jq^2}+O\left(\frac{1}{d_{1,j}q^3}\right)\right.\\
			&\left.-\frac{d_{1,j}^2C_j^2}{24}\left(1+\frac{u_2^*}{d_{1,j}C_jq^2}+O\left(\frac{1}{d_{1,j}q^3}\right)\right)^3+O\left(d_{1,j}^4\right)\right]^6\\
			&=-C_j^6d_{1,j}^6\left[1+\frac{u_2^*}{d_{1,j}C_jq^2}+O\left(\frac{1}{d_{1,j}q^3}\right)+O\left(d_{1,j}^4\right)\right]^6\\
			&=-C_j^6d_{1,j}^6\left[1+\frac{6u_2^*}{d_{1,j}C_jq^2}+O\left(\frac{1}{d_{1,j}q^3}\right)+O\left(d_{1,j}^4\right)\right]\\
			&=-C_j^6d_{1,j}^6-\frac{6u_2^*C_j^5d_{1,j}^5}{q^2}+O\left(\frac{d_{1,j}^5}{q^3}\right)+O\left(d_{1,j}^{10}\right)\ \ (q\to\infty).
		\end{align*}
		
		Finally, taking into account that $C_j=1+\frac{u_1^*}{\pi j}$ and $q=\frac{n+3}{2}$, we obtain
		\[
		\lambda_{2j-1}^{(n)}=-\frac{(2\pi j+2u_1^{*})^6}{(n+3)^6}-\frac{24u_2^*(2\pi j+2u_1^*)^5}{(n+3)^7}+\Delta_1(n,j)\ \ (n\to\infty).
		\]
		\item In a similar fashion, given that $\lambda_{2j}^{(n)}=g(\varphi_{2j}^{(n)})$ it is possible to deduce that 
		\[\lambda_{2j}^{(n)}=-\frac{((2j+1)\pi+2w_1^{*})^6}{(n+3)^6}-\frac{24w_2^*((2j+1)\pi+2w_1^*)^5}{(n+3)^7}+\Delta_2(n,j),\ \ (n\to\infty).
		\]
	\end{enumerate}
\end{proof}
Now, we prove the equivalence between the asymptotic formula presented in \cite{Parter1961} with our result in Theorem \ref{extreme eigenvalues}.

In \cite{Parter1961} the author considered the class of functions $g$ satisfying:
\begin{enumerate}[$(a)$]
	\item $g$ is real, continuous, and periodic with period $2\pi$; $\min g=g(0)=m^*$ and $\varphi=0$ is the only value of $\varphi\ (mod\ 2\pi)$ for which this minimum is attained.
	\item If $g$ satisfies (a), then it has continuous derivatives of order $2k$ ($k\in\mathbb{N}$) in some neighborhood of $\varphi=0$ and $g^{(2k)}(0)=\sigma^2>0$ is the first non-vanishing derivative of $g$ at $\varphi=0$.
\end{enumerate}

\begin{theorem}[\cite{Parter1961} Theorem 4]\label{Parter}
	Let $g$ be a function which satisfies Conditions (a) and (b). Let $\lambda_{1,n}$ be the minimal eigenvalue of $T_n(a)$. Then
	\[
	\lambda_{1,n}=m^*+O(n^{-2k}),
	\]
	where $O$ cannot be replaced by $o$.
\end{theorem}

Let $g_1(\varphi)=-g(\varphi)=(2\sin\frac{\varphi}{2})^6$. Notice that $g_1$ satisfies Conditions $(a)$ and $(b)$ with $m^*=0$ and $g^{(6)}(0)=720>0$. Therefore, from Theorem \ref{Parter} we get
\begin{align}\label{Parter'sAsymptotic}
	\lambda_{1,n}=m^*+O\left(\frac{1}{n^6}\right)\ \ (n\to\infty).
\end{align} 

On the other hand, if $j=1$ in Theorem \ref{extreme eigenvalues}, then we easily obtain the following asymptotic expansion 
\begin{align*}
	\lambda_{1}^{(n)}=m^*+O\left(\frac{1}{n^6}\right)\ \ (n\to\infty),
\end{align*}
which coincides with Eq. \eqref{Parter'sAsymptotic}.
\section{Numerical experiments}
The experiments in this section will be carried out using the Maple mathematical package. The first graph \ref{fig:fig1} shows the dependence of the relative error of the eigenvalue on the iteration number in the formulas of the theorem \ref{theorem 2}. The size of the matrix is $ 200 \times 200$, the error is calculated for the first, average and last eigenvalues (ordered by modulus). Here $k$ is the number of iterations and $m$ is the number of the eigenvalue.
\begin{figure}[!h]
\includegraphics[width=\textwidth]{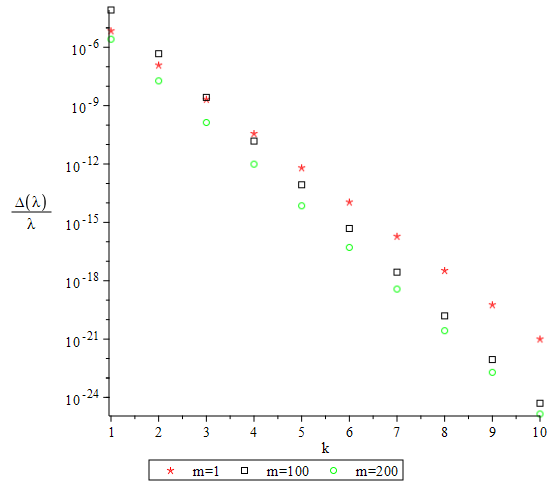}
	\caption{Dependence of the error on the number of iterations}
		\label{fig:fig1}
\end{figure}
The same results are shown in the table \ref{tab1}\\
\begin{table}[!h]
		\center
		\caption{Dependence of the error on the number of iterations}
		\label{tab1}
\begin{tabular}{|c|c|c|c|c|c|}
	\hline
	k&	1&	2&	3&	4&	5\\ \hline
	$\frac{\delta \lambda}{\lambda}$&	$1.33 \cdot 10^{-4}$&		$1.9 \cdot 10^{-5}$&		$2.55 \cdot 10^{-6}$&		$3.31 \cdot 10^{-7}$&		$4.21 \cdot 10^{-8}$ \\ \hline
	m=100 &	$8.28 \cdot 10^{-5}$&		$4.7 \cdot 10^{-7}$&		$2.66 \cdot 10^{-9}$&		$1.51 \cdot 10^{-11}$&		$8.55 \cdot 10^{-14}$	\\ \hline
	m=200 &	$2.61 \cdot 10^{-6}$&		$1.88 \cdot 10^{-8}$&		$1.36 \cdot 10^{-10}$&		$9.28 \cdot 10^{-13}$&		$7.1 \cdot 10^{-15}$\\ \hline
		k&	6&	7&	8&	9&	10\\ \hline
	$\frac{\delta \lambda}{\lambda}$&		$1.11 \cdot 10^{-14}$&		$1.92 \cdot 10^{-16}$&		$3.33 \cdot 10^{-18}$&		$5.78 \cdot 10^{-20}$&    $1 \cdot 10^{-21}$ \\ \hline
	m=100 &		$4.85 \cdot 10^{-16}$&		$2.75 \cdot 10^{-18}$&		$1.56 \cdot 10^{-20}$&		$8.83 \cdot 10^{-23}$&		$5.01 \cdot 10^{-25}$\\ \hline
	m=200 &	    $5.13 \cdot 10^{-17}$&	    $3.71 \cdot 10^{-19}$&		$2.68 \cdot 10^{-21}$&		$1.94 \cdot 10^{-23}$&		$1.4 \cdot 10^{-25}$\\ \hline
\end{tabular}
\end{table}
Table \ref{tab2} shows the maximum relative errors when using the formula from Theorem \ref{Teorema sluch 1}. The maximum was considered for all eigenvalues starting from the seventh. In Table \ref{tab3} and Table \ref{tab3} $n$ is matrix size.\\
Table \ref{tab3} shows the maximum relative deviations when using the formulas from Theorem \ref{theorem sluch 2}. To find $ u_1^\star$ and $ w1^\star$ a recursive formula was used. The number of iterations was taken equal to four. The maximum was found for the first six eigenvalues. 
\begin{table}[h!]
	\center
	\caption{Maximum relative error when using the formula from Theorem \ref{Teorema sluch 1} }
	\label{tab2}
	\begin{tabular}{|l|l|l|l|l|l|}
		\hline
		n&	32&	64&	128&	256&	512\\ \hline
	$\frac{\Delta \lambda}{\lambda}$&	$1.33 \cdot 10^{-4}$&		$1.9 \cdot 10^{-5}$&		$2.55 \cdot 10^{-6}$&		$3.31 \cdot 10^{-7}$&		$4.21 \cdot 10^{-8}$\\ \hline
	\end{tabular}
\end{table}
\begin{table}[h!]
	\center
	\caption{Maximum relative error when using the formula from Theorem \ref{theorem sluch 2} }
	\label{tab3}
	\begin{tabular}{|l|l|l|l|l|l|}
		\hline
		n&	32&	64&	128&	256&	512\\ \hline
		$\frac{\Delta \lambda}{\lambda}$&	$4.17 \cdot 10^{-4}$&		$2.94 \cdot 10^{-5}$&		$1.97 \cdot 10^{-6}$&		$1.28 \cdot 10^{-7}$&		$8.19 \cdot 10^{-9}$\\ \hline
	\end{tabular}
\end{table}

\section{Acknowledgment}
This work is funded by RSCF-21-11-00283\\

\normalsize
\bibliography{bibl}

\begin{thebibliography}{10}

\bibitem{GS}
Leonard~J. Savage, Ulf Grenander, and Gabor Szego.
\newblock Toeplitz forms and their applications.
\newblock {\em Journal of the American Statistical Association}, 53(283):763,
  sep 1958.

\bibitem{SS}
Palle Schmidt and Frank Spitzer.
\newblock The toeplitz matrices of an arbitrary laurent polynomial.
\newblock {\em {MATHEMATICA} {SCANDINAVICA}}, 8:15, dec 1960.

\bibitem{WiOT}
Harold Widom.
\newblock Eigenvalue distribution of nonselfadjoint toeplitz matrices and the
  asymptotics of toeplitz determinants in the case of nonvanishing index.
\newblock {\em Oper. Theory Adv. Appl.}, 48, 01 1990.

\bibitem{BG}
Albrecht Böttcher and Sergei~M. Grudsky.
\newblock {\em Spectral Properties of Banded Toeplitz Matrices}.
\newblock Society for Industrial and Applied Mathematics, jan 2005.

\bibitem{Uni}
Albrecht Böttcher and Bernd Silbermann.
\newblock {\em Introduction to Large Truncated Toeplitz Matrices}.
\newblock Springer New York, 1999.

\bibitem{DIKIsing}
Percy Deift, Alexander Its, and Igor Krasovsky.
\newblock Toeplitz matrices and toeplitz determinants under the impetus of the
  ising model: Some history and some recent results.
\newblock {\em Communications on Pure and Applied Mathematics},
  66(9):1360--1438, jun 2013.

\bibitem{DIK}
P.~Deift, A.~Its, and I.~Krasovsky.
\newblock Eigenvalues of toeplitz matrices in the bulk of the spectrum.
\newblock {\em Bull. Inst. Math. Acad. Sin. (N. S.) 7 (2012), 437-461}, October
  2011.

\bibitem{Kad}
L.~P. Kadanoff.
\newblock Spin-spin correlations in the two-dimensional ising model.
\newblock {\em Il Nuovo Cimento B Series 10}, 44(2):276--305, aug 1966.

\bibitem{MW}
B.~McCoy and T.~Wu.
\newblock The two-dimensional ising model.
\newblock 1973.

\bibitem{BGS}
A.A. Batalshchikov, S.M. Grudsky, and V.A. Stukopin.
\newblock Asymptotics of eigenvalues of symmetric toeplitz band matrices.
\newblock {\em Linear Algebra and its Applications}, 469:464--486, 2015.

\bibitem{BBG}
J.~M. Bogoya, A.~Böttcher, and S.~M. Grudsky.
\newblock Asymptotics of individual eigenvalues of a class of large hessenberg
  toeplitz matrices.
\newblock In {\em Recent Progress in Operator Theory and Its Applications},
  pages 77--95. Springer Basel, 2012.

\bibitem{BBGM1}
J.M. Bogoya, A.~Böttcher, S.M. Grudsky, and E.A. Maximenko.
\newblock Eigenvalues of hermitian toeplitz matrices with smooth simple-loop
  symbols.
\newblock {\em Journal of Mathematical Analysis and Applications},
  422(2):1308--1334, feb 2015.

\bibitem{Grudsky2011}
A.~Böttcher, S.M. Grudsky, and E.A. Maksimenko.
\newblock Inside the eigenvalues of certain hermitian toeplitz band matrices.
\newblock {\em Journal of Computational and Applied Mathematics},
  233(9):2245--2264, mar 2010.

\bibitem{BGM1}
J.~M. Bogoya, S.~M. Grudsky, and E.~A. Maximenko.
\newblock Eigenvalues of hermitian toeplitz matrices generated by simple-loop
  symbols with relaxed smoothness.
\newblock In {\em Large Truncated Toeplitz Matrices, Toeplitz Operators, and
  Related Topics}, pages 179--212. Springer International Publishing, 2017.

\bibitem{Bat2019}
A.A. Batalshchikov, S.M. Grudsky, I.S. Malisheva, S.S. Mihalkovich,
  E.~{Ramírez de Arellano}, and V.A. Stukopin.
\newblock Asymptotics of eigenvalues of large symmetric toeplitz matrices with
  smooth simple-loop symbols.
\newblock {\em Linear Algebra and its Applications}, 580:292--335, 2019.

\bibitem{BarreraG}
M.~Barrera and S.~M. Grudsky.
\newblock Asymptotics of eigenvalues for pentadiagonal symmetric toeplitz
  matrices.
\newblock In {\em Large Truncated Toeplitz Matrices, Toeplitz Operators, and
  Related Topics}, pages 51--77. Springer International Publishing, 2017.

\bibitem{Elou}
Mohamed Elouafi.
\newblock On a relationship between chebyshev polynomials and toeplitz
  determinants.
\newblock {\em Applied Mathematics and Computation}, 229:27–33, 02 2014.

\bibitem{Parter1961}
Seymour~V. Parter.
\newblock Extreme eigenvalues of toeplitz forms and applications to elliptic
  difference equations.
\newblock {\em Transactions of the American Mathematical Society},
  99(1):153--153, jan 1961.

\end{thebibliography}

\end{document}